\documentclass[reqno,11pt]{amsart}

\usepackage[latin1]{inputenc}
\usepackage[english]{babel}

\usepackage{amscd}

\usepackage{amsmath}
\usepackage{amsthm}
\usepackage{amssymb}

\usepackage[notref, notcite, final]{showkeys}
\usepackage{color}
\usepackage{graphicx}
\usepackage[hidelinks]{hyperref}
\usepackage{cleveref}
\usepackage{fullpage}
\usepackage{comment}
\usepackage[shortlabels]{enumitem}

\usepackage{dsfont}

\newcommand{\id}{\mathcal{I}}

\newcommand{\rr}{\mathbb{R}}
\newcommand{\cc}{\mathbb{C}}

\newcommand{\Q}{\mathcal{Q}}
\renewcommand{\Re}{\mathrm{Re}}

\newcommand{\sderiv}[1][]{\partial_{S#1}}

\newcommand{\boundOP}{\mathcal{B}}

\newcommand{\dom}{\operatorname{dom}}
\newcommand{\dist}{\operatorname{dist}}

\usepackage[autostyle]{csquotes}

\newtheorem{theorem}{Theorem}[section]
\newtheorem{lemma}[theorem]{Lemma}
\newtheorem{proposition}[theorem]{Proposition}

\theoremstyle{definition}
\newtheorem{assumption}[theorem]{Assumption}

\theoremstyle{definition}
\newtheorem{definition}[theorem]{Definition}
\newtheorem{example}[theorem]{Example}

\theoremstyle{remark}
\newtheorem{remark}[theorem]{Remark}
\crefname{enumi}{}{}

\title{\bf Universality property of the $S$-functional
calculus, noncommuting matrix variables and Clifford operators }

\author[F. Colombo]{Fabrizio Colombo}
\address{(FC) Politecnico di
Milano\\Dipartimento di Matematica\\Via E. Bonardi, 9\\20133 Milano,
Italy}
\email{fabrizio.colombo@polimi.it}

\author[J. Gantner]{Jonathan Gantner}
\address{(JG)
(was PhD student) Politecnico di Milano\\Dipartimento di Matematica\\Via E. Bonardi, 9\\20133
Milano, Italy
} \email{jonathan.gantner@gmx.at}

\author[D. P. Kimsey]{David P. Kimsey}
\address{(DPK) School of Mathematics, Statistics and Physics, Newcastle University, Newcastle upon Tyne NE1 7RU UK
} \email{david.kimsey@ncl.ac.uk}

\author[I. Sabadini]{Irene Sabadini}
\address{(IS) Politecnico di
Milano\\Dipartimento di Matematica\\Via E. Bonardi, 9\\20133 Milano\\Italy}
\email{irene.sabadini@polimi.it}

\begin{document}
\maketitle

\begin{abstract}
Spectral theory on the $S$-spectrum was born out of the need
to give quaternionic quantum mechanics a precise mathematical foundation (Birkhoff and von Neumann \cite{BF} showed that a general set theoretic formulation of quantum mechanics can be realized on real, complex or quaternionic Hilbert spaces).
Then it turned out that spectral theory on $S$-spectrum has important applications in several fields such as
fractional diffusion problems and, moreover,  it allows one to define several functional calculi
for $n$-tuples of noncommuting operators.
With this paper we show that the spectral theory on the $S$-spectrum is much more general and
it contains, just as particular cases,
the complex, the quaternionic and the Clifford settings.
In fact, the $S$-spectrum is well defined for
objects in an algebra that has a complex structure and for operators in general Banach modules. We show that the abstract formulation of the $S$-functional calculus
 goes beyond quaternionic and Clifford analysis, indeed the $S$-functional calculus has a certain {\em universality property}.
This fact makes the spectral theory on the $S$-spectrum
applicable to several fields of operator theory and allows one to define functions of
noncommuting matrix variables, and operator variables, as a particular case.
\end{abstract}

\medskip
\noindent AMS Classification  47A10, 47A60.


\date{today}
\tableofcontents

\section{Introduction}

An important issue in operator theory is to
 define functions of $n$-tuples of operators $(T_1,...,T_n)$ and in order to do this several different strategies have been developed.
The spectral theorem for $n$-tuples of commuting normal operators on a Hilbert space, see \cite{Schmuedgen},
 and the Weyl functional calculus for self adjoint, not necessarily commuting operators,
  are among the most important tools.
 The use of the Cauchy formulas of  hyperholomorphic functions
 constitutes a powerful
  strategy to define functions of quaternionic operators and also of $n$-tuples of operators $(T_1,...,T_n)$ and these hyperholomorphic functional calculi have several applications in
  mathematics, physics and engineering.
  In  \cite{jefferies} one can find connections of the Weyl functional calculus and the Taylor functional calculus (see \cite{Taylor}) with the monogenic functional calculus. Additional connections among the aforementioned spectral theories can be found in the survey paper \cite{CGP}.
  The holomorphic functional calculus can be extended to unbounded operators and
for sectorial operators the $H^\infty$-functional calculus, introduced by A. McIntosh in \cite{McIntosh:1986}, is a very important extension, see also \cite{Haase}.

\medskip
There are two different extensions to higher dimensions of holomorphic functions of one complex variable obtained by the Fueter-Sce-Qian theorem. These extensions give two different notions of hyperholomorphic functions, see \cite{bookSCE},
 connected with quaternionic-valued functions
or more in general with Clifford algebra-valued functions.
The two extensions are called
slice hyperholomorphic functions and monogenic functions for the Clifford algebra setting.
Both classes of hyperholomorphic functions have a Cauchy formula that can be used to define
functions of $n$-tuples of operators that do not necessarily commute.

\medskip
Slice hyperholomorphic functions with values in a  Clifford algebra are also called
 slice monogenic functions. The main results of this function theory are contained in the book \cite{6css} and the references therein.
Monogenic functions, i.e.,
functions that are in the kernel of the Dirac operator,
see \cite{bds,BLU,Gurlebeck:2008}, are very well studied.

\medskip
The two spectral theories induced by hyperholomorphic functions are the spectral theory based on the $S$-spectrum and the monogenic spectral theory based on the monogenic spectrum.
Precisely,
the Cauchy formula of slice hyperholomorphic functions generates the $S$-functional calculus
 for $n$-tuples of not necessarily commuting operators, this calculus is based on the  notion of
$S$-spectrum.
On the other hand, the Cauchy formula, that monogenic functions generate, gives rise to the monogenic functional calculus which is based on the monogenic spectrum.

\medskip
When considering intrinsic functions, the $S$-functional calculus can be defined for one-sided quaternionic Banach spaces, see \cite{SpectralOP}.
In the paper \cite{SpectralOP} the author  has also developed the theory of spectral operators in quaternionic Banach spaces.
The two hyperholomorphic functional calculi coincide with the  Riesz-Dunford functional calculus
when they are applied to a single operator on a real or complex Banach space.

\medskip
It is a well known fact that in the complex setting the spectral theorem and the Riesz-Dunford functional calculus are both based on the same notion of spectrum and this is due to the fact there do not seem to be any other useful and meaningful notions of spectrum in this case.
If we move to the quaternionic setting we have two ways to generalize the eigenvalue problem,
considering left and right eigenvalues.
In both cases these spectra are not suitable for both the quaternionic holomorphic functional calculus
and for the quaternionic spectral theorem.
It was just in 2006 with the discovery of the $S$-spectrum and with the theory of slice hyperholomorphic functions over the quaternions that the symmetry became clear. The $S$-functional calculus and the quaternionic spectral theorem are both based on the notion of $S$-spectrum.

\medskip
The spectral theorem based on the notion of $S$-spectrum for quaternionic bounded and unbounded normal operators has been proved in 2015 and published in the paper \cite{ack} that
 appeared in 2016.
As in the complex case, the spectral theorem for unitary operators, see \cite{6SpectThmIrene}, can be deduced by the quaternionic version of Herglotz's theorem \cite{Herglotz}.
Perturbation of normal quaternionic operators are studied in \cite{6CCKS}.
Beyond the spectral theorem there is the theory of the characteristic operator functions that has been initiated in \cite{6COFBook}.

In the past there have been several attempts to generalize the spectral theory to a quaternionic Hilbert space, see \cite{Teichmueller,Viswanath},
 but without specifying the notion of quaternionic spectrum. The papers contain some interesting results in quaternionic operators theory.

 \medskip
 In 2020, the spectral theorem  for  Clifford normal operators based on the $S$-spectrum was proved in  \cite{CLIFSPECTAL}, but with some surprise this theorem holds true when we consider full Clifford operators and not only paravector operators.

 \medskip
 This fact has lead the authors to delve deeper into the $S$-functional calculus in order to restore the lost symmetry. Indeed, it turns out that we can extend the $S$-functional calculus to full Clifford operators and not only to paravector operators.
 This observation  has important consequences also on the function theory of slice monogenic functions
 because it opens the way to the definition of slice monogenic functions of a Clifford variable,
 see \cite{CKPS}.

\medskip
 The literature on hyperholomorphic function theories and related spectral theories is nowadays very large.
 For the function theory of slice hyperholomorphic functions, started with the quaternionic case in \cite{gs}, the
 main books are \cite{6EntireBook,6css,6GSBook,6GSSbook}, while
 for the spectral theory on the $S$-spectrum we mention the books \cite{6COFBook,6CG,6CKG,6css}.
The monogenic functional calculus can be found in the book
\cite{jefferies} and the references therein.

\medskip

Let us explain the generality of the spectral theory on the $S$-spectrum.
The point of view that we adopt in this paper is to work with a nontrivial, real, unital algebra $\mathcal F$ which is  finite-dimensional, associative, and is equipped with an anti-involution and a norm.

\begin{remark}
The framework in which we work is rather general since
the algebra $\mathcal{F}$ that we consider here is not necessarily of dimension $2^n$,
so it is not necessarily isomorphic to a Clifford algebra,
but the family of algebras that we consider does include
the algebra of complex numbers, quaternions and the  Clifford
algebras $\mathbb R_{0,n}$  with $n\ge 1$.
\end{remark}

We denote by $\mathbb{S}$ the set of left
imaginary units in $\mathcal{F}$, i.e.,
\begin{equation*}
         {\mathbb S}:=\left\{s\in \mathcal{F}\ :\ s\bar{s}\in \mathbb{R} \ {\rm and}\ \
s^2=-1_{\mathcal{F}}\right\},
\end{equation*}
where $1_{\mathcal{F}}$ denotes the unit in ${\mathcal{F}}$, later on denoted by $1$ for simplicity, and $\bar{\cdot}$ is the anti-involution.
When $\mathbb S\not=\emptyset$, we define the (left) weak cone of $\mathcal{F}$:
     \begin{equation*}
         \mathcal{W}_{\mathcal{F}}:=
         \bigcup_{\mathrm{J}\in\mathbb{S}}\mathbb{C}_\mathrm{J},
     \end{equation*}
where $\mathbb{C}_\mathrm{J}$ will be called complex plane associated with $\mathrm{J}\in\mathbb{S}$.
For $s\in \mathcal{W}_{\mathcal{F}}$, using the anti-involution, we define the real numbers
${\rm Re}(s):=\frac{1}{2}(s+\bar s)$  and $|s|:=\sqrt{s \bar s}$, see in the sequel for more details.
Now we consider a  two-sided
  Banach module, denoted by $\mathcal{Y}$,  over $\mathcal{F}$ with norm $\|\cdot\|_\mathcal{Y}$.
We denote by $\mathcal{B}(\mathcal{Y})$ the Banach module of all bounded linear operators from
$\mathcal{Y}$ into itself with the natural operator norm
and, mimicking the analogous definitions in \cite{6css}, we give the definition of $S$-resolvent set and of $S$-spectrum with greater generality than \cite{6css}. In \cite{GR} the authors considered these notions for analysis of semigroups.
Let $A\in \mathcal{B}(\mathcal{Y})$ and $s\in \mathcal{W}_{\mathcal{F}}$, we set
\[
\Q_{s}(A):=A^2 - 2\Re(s)A + |s|^2\id,
\]
where $\mathcal{I}$ is the identity operator in $\mathcal{B}(\mathcal{Y})$.
We define the $S$-resolvent set $\rho_S(A)$ of $A$ as
\[
\rho_S(A) := \{s\in \mathcal{W}_{\mathcal{F}} : \Q_{s}(A)  \text{ is invertible in $\mathcal{B}(\mathcal{Y})$}\}
\]
and the $S$-spectrum $\sigma_S(A)$ of $A$ as
\[
 \sigma_S(A) := \mathcal{W}_{\mathcal{F}}\setminus\rho_S(A).
\]
Then we define the {\em left $S$-resolvent operator} as
$$
 S_L^{-1}(s,A) = -\Q_{s}(A)^{-1}(A-\overline{s}\,\id), \ \ \ \ s\in \rho_S(A)
 $$
and the {\em right $S$-resolvent operator} as
$$
S_R^{-1}(s,A) = -(A-\overline{s}\id)\Q_{s}(A)^{-1}, \ \ \ \ s\in \rho_S(A),
$$
The operator $\Q_{s}(A)^{-1}$, for $s\in \rho_S(A)$, is called pseudo-resolvent operator.
Once that we have introduced the above concepts
 we can define the $S$-functional calculus with a great generality in such a way  that it contains as  particular cases: the complex case, the quaternionic case  and the case of paravector operators or full Clifford operators.
So we denote by $\mathcal{SH}_L(\sigma_{S}(A))$,
(resp. $\mathcal{SH}_R(\sigma_{S}(A))$)
for $A\in\boundOP(\mathcal{Y})$,
the set of all  left, (resp. right)   slice hyperholomorphic  functions on $U$,  where $U$ is a bounded Cauchy domain such that with $\sigma_{S}(A)\subset U$,
$\overline{U}\subset\dom(f)$ and $\dom(f)$ is the domain of the function $f$.
The abstract formulation of the $S$-functional calculus holds
for $A\in\boundOP(\mathcal{Y})$. For any imaginary unit $\mathrm{J}\in\mathbb{S}$ we set $ds_\mathrm{J}=ds(-\mathrm{J})$ and we define
\begin{equation*}
f(A) := \frac{1}{2\pi}\int_{\partial(U\cap\mathbb{C}_\mathrm{J})}S_L^{-1}(s,A)\,ds_\mathrm{J}\,f(s),
 \ \ {\rm for \ any} \ \  f\in\mathcal{SH}_L(\sigma_{S}(A))
\end{equation*}
and
\begin{equation*}
f(A) := \frac{1}{2\pi}\int_{\partial(U\cap\mathbb{C}_\mathrm{J})}f(s)\,ds_\mathrm{J}\,S_R^{-1}(s,A),
 \ \ {\rm for \ any} \ \  f\in\mathcal{SH}_R(\sigma_{S}(A)).
\end{equation*}
Moreover, standard techniques in slice hypercomplex analysis, see \cite{6CKG,6css}, show that both  integrals  depend neither on $U$ nor on the imaginary unit $\mathrm{J}\in\mathbb{S}$.
The arbitrariness  of the
 algebra $\mathcal{F}$ and of  the
Banach module $\mathcal{Y}$  over $\mathcal{F}$ is called the {\em universality property} of the $S$-functional calculus.

\medskip
The paper consists of six sections including the introduction.
In section 2 we consider vector-valued slice hyperholomorphic functions on the algebra $\mathcal{F}$
which are the hyperholomorphic functions on which the $S$-functional calculus is based.
In Section 3 we give the abstract formulation $S$-functional calculus
and in Section 4 we apply the abstract results to bounded full Clifford operators.
In Section 5 we show some properties of the $S$-functional calculus for bounded full Clifford operators and finally in Section 6 we consider the particular case of noncommuting matrix variables which can be of interest in free probability.

\section{Vector-valued slice hyperholomorphic functions}\label{UNIVERS}

In this section we define vector-valued slice hyperholomorphic functions with values in a module over an associative algebra satisfying suitable conditions. In the sequel, we use the notion of hyperholomorphicity to define the $S$-functional calculus for a large class of operators to which the definition of $S$-spectrum can be extended and for which the $S$-resolvent operators preserve the slice hyperholomorphicity.

Let $\mathcal{F}\neq\{0\}$ be a real, finite-dimensional, associative real algebra with unit, denoted by $1$. The multiples of the unit $1$ will be identified with real numbers by $\alpha \cdot 1\mapsto \alpha$, $\alpha\in\mathbb R$.
By fixing a basis $u_1=1,u_2,\dots , u_N$ of $\mathcal{F}$ as a real linear space we can write $a\in\mathcal{F}$ as $a=\sum_{\ell=1}^N a_\ell u_\ell$, $a_\ell\in\mathbb R$. We say that $a_1$ is the so-called scalar part of $a$ and, following the notation in use in the case of Clifford algebras, we denote it by $[a]_0$.\\
  We assume that $\mathcal{F}$ is equipped with an anti-involution $\bar{\cdot}$, i.e., $\bar{\cdot}$ is an involution in $\mathcal{F}$ with the property $\overline{ab}=\bar b\bar a$ for all $a,b\in\mathcal{F}$, that fixes $\mathbb{R}$.
\\ In the sequel, we need a notion of norm in $\mathcal{F}$, so we assume that the algebra $\mathcal{F}$ satisfies the condition $[a\bar a]_0\geq 0$, for $a\in \mathcal{F}$,  and the equality holds if and only if $a=0$. Then we define
\begin{equation}\label{norm}
\|a\|=([a\bar a]_0)^{1/2}
\end{equation}
and we assume that
$\|\cdot \|$ defines a norm in $\mathcal{F}$.
\medskip

We denote by $\mathbb{S}$ the set of
imaginary units in $\mathcal{F}$, i.e.,
\begin{equation}\label{eq-ca}
         {\mathbb S}:=\left\{s\in \mathcal{F}\ :\ s\bar{s}\in \mathbb{R} \ {\rm and}\ \
s^2=-1 \right\}.
\end{equation}
We shall assume that $\mathcal{F}$ is such that $\mathbb{S}\not=\emptyset$.
\\
 If we denote by $L_s:\
\mathcal{F}\rightarrow \mathcal{F}$ the multiplication on the left
by $s\in \mathcal{F}$, our assumptions imply that there exists $s\in \mathcal{F}$ such that $L_s$ is a complex structure on $\mathcal{F}$. So the algebra $\mathcal{F}$ is a LSCS algebra, in the terminology of \cite{dou}.
\begin{assumption}\label{ASS}
  In the sequel $\mathcal{F}$  denotes a nontrivial, finite-dimensional, associative  real algebra with unit, nonempty $\mathbb{S}$ and equipped with an anti-involution $\bar{\cdot}$ giving a norm $\|\cdot\|$ as in \eqref{norm}.
\end{assumption}

Following \cite{dou}, we now give the following definition:
\begin{definition} \label{weakcone}
We define the (left) weak cone of $\mathcal{F}$:
     \begin{equation}\label{cone}
         \mathcal{W}_{\mathcal{F}}:=
         \bigcup_{\mathrm{J}\in\mathbb{S}}\mathbb{C}_\mathrm{J},
     \end{equation}
where $\mathbb{C}_\mathrm{J}:=\{u+\mathrm{J}v, \ u,v\in \mathbb{R}\}$ will be called complex plane associated with $\mathrm{J}\in\mathbb{S}$.
\end{definition}
\begin{example}
In the case we consider $\mathcal{F}=\mathbb{H}$, we have $\mathcal{W}_{\mathcal{F}}=\mathbb{H}$.
For a Clifford algebra $\mathcal{F}=\mathbb{R}_n$ the set of paravectors
$\{s\in \mathbb{R}_n \  :  \ s=s_0+s_1e_1+...+s_ne_n\}$ is contained in
$\mathcal{W}_{\mathcal{F}}$ but $\mathcal{W}_{\mathcal{F}}$ can be larger, see \cite{GP}.
\end{example}

In \cite{GP} the authors define the notion of a quadratic cone in a real, finite-dimensional, alternative algebra with unit $1$ and equipped with an anti-involution. The imaginary units are a subset of this cone which turns out to satisfy \eqref{cone} for $\mathrm{J}\in{\mathbb S}$.
\begin{remark}
For any $\mathrm{J}\in\mathbb S$ we have, by definition, $\mathrm{J}^2=-1$ and so $\|\mathrm{J}^2\|=1$. The definition of norm implies
$\|\mathrm{J}^2\|=(\mathrm{J}\mathrm{J})\overline{(\mathrm{J}\mathrm{J})}
=\mathrm{J}(\mathrm{J}\bar{\mathrm{J}})\bar{\mathrm{J}}=\|\mathrm{J}\|^2=1$ thus $\|\mathrm{J}\|=1$, i.e. $\mathrm{J}\bar{\mathrm{J}}=1$. Multiplying by $\mathrm J$ both sides of this last equality we obtain $-\bar{\mathrm{J}}=\mathrm{J}$.
We deduce that the anti-involution $\bar{\cdot}$ induces the standard conjugation on each complex plane $\mathbb C_\mathrm{J}$, namely $\bar{\mathrm{J}}=-\mathrm{J}$ and
$$
\overline{\cdot}: \ \mathcal{W}_{\mathcal{F}}\to \mathcal{W}_{\mathcal{F}},
$$
which is defined by $s=u+\mathrm{J}v\mapsto \bar s =u-\mathrm{J}v$, for $u,v\in \mathbb{R}$.
Moreover,  for every $s\in\mathcal{W}_{\mathcal{F}}$, note that:
$$
s+\bar s \in \mathbb R, \qquad s \bar s=\bar s  s \in\mathbb R,
$$
and so we set
$$
{\rm Re}(s):=\frac{1}{2}(s+\bar s), \ \ \ {\rm and} \ \  |s|:=\sqrt{s \bar s}
$$
that we will call the real part and the modulus of $s$, respectively, in analogy with the complex case. Observe also that $|s|=\|s\|$.
\end{remark}
\begin{lemma}\label{squadequat}
For any $s\in \mathcal{W}_{\mathcal{F}}$ we have the identity
$$
s^2-2s {\rm Re}(s)+|s|^21=0.
$$
\end{lemma}
\begin{proof}
The proof is immediate, since any $s\in \mathcal{W}_{\mathcal{F}}$
 satisfies
$$
s^2-s(s+\bar s) +s\bar s 1 =0.
$$
\end{proof}

\begin{example}
The algebra of quaternions satisfies Assumption \ref{ASS}. If we consider the elements $i,j$ such that $i^2=j^2=-1$, $ij+ji=0$, the set $1,i,j, ij$ is a basis of $\mathbb H$ as a real linear space and the conjugation is $\bar i=-i$, $\bar j=-j$, $\overline{ij}=-ij$. The norm is such that $\|a\|=(a\bar a)^{1/2}$. Moreover $\mathbb H=\mathcal{W}_{\mathcal{F}}$.
\\
More generally, a Clifford algebra $\mathbb{R}_n := \mathbb{R}_{0,n}$, $n\geq 1$ satisfies Assumption \ref{ASS}, if we consider the Clifford conjugation.
\\
Finally we consider the real algebra $\mathbb{BC}$  of bicomplex numbers (which is not a real Clifford algebra). It is a commutative algebra generated by two imaginary units $i,j$ such that $i^2=j^2=-1$, $ij=ji$. Simple computations show that $\mathbb{S}=\{\pm i,\pm j\}$ and $\mathcal W_{\mathbb{BC}}=\mathbb{C}_i \cup \mathbb{C}_j$. The conjugation $\bar i=-i$, $\bar j=-j$ extends to the whole algebra and $\overline{ab}=\bar b\bar a$. Writing $a=x+yi+zj+tk$, where $k=ij=\bar k$, we have that $\|a\|=(x^2+y^2+z^2+t^2)^{1/2}$, defined according to \eqref{norm}, is the Euclidean norm.
\end{example}


\begin{definition}
Let  $x=u+\mathrm{J}_{x}v\in  \mathcal{W}_{\mathcal{F}}$,
 for $u,v\in \mathbb{R}$. We define the set
$$
[x]:=\{y\in \mathcal{W}_{\mathcal{F}}\ :\ y=u+\mathrm{J} v, \
 \mathrm{J}\in\mathbb{S}\},
$$
which will be called {\it the sphere associated with} $x$.
We say that $U\subseteq  \mathcal{W}_{\mathcal{F}}$ is {\it axially symmetric}
 if $[x]\in U$  for any $x\in U$.
\end{definition}

Since we need to consider functions defined on subsets of the weak cone $\mathcal{W}_{\mathcal{F}}$, our next goal is to define a topology on it and to this end we shall follow \cite{dou}.
\begin{definition}
The slice topology $\tau_s(\mathcal{W}_{\mathcal{F}})$ on $\mathcal{W}_{\mathcal{F}}$ is defined
by
$$
\tau_s(\mathcal{W}_{\mathcal{F}})=\{U\subset \mathcal{W}_{\mathcal{F}}\ :\ U\cap\mathbb C_{\mathrm{J}} \ \text{ is\ open\ in\ } \mathbb{C}_{\mathrm{J}}, \forall {\mathrm{J}}\in\mathbb S \}.
$$
\end{definition}
The weak cone can be equipped with the slice topology, see \cite{dou}. Among the open sets in $\mathcal{W}_{\mathcal{F}}$, i.e.  sets in $\tau_s(\mathcal{W}_{\mathcal{F}})$, there are the axially symmetric open sets. These sets are of the form
\begin{equation}\label{axsymm}
U=\{x\in \mathcal{W}_{\mathcal{F}}\ :\ x=u+\mathrm{J} v, \ (u,v)\ \in\mathcal{U}, \,
 \mathrm{J}\in\mathbb{S}\},
\end{equation}
where $\mathcal{U}\subseteq\mathbb R^2$ is an open set in $\mathbb R^2$. In the sequel, we shall consider only open sets which are axially symmetric, namely axially symmetric sets generated by open sets $\mathcal U$ in $\mathbb R^2$.

Our next goal is to define
vector-valued slice hyperholomorphic functions with values in a left (or right, or two-sided) Banach module over $\mathcal{F}$. We start by considering a special case in which the Banach module $\mathcal{X}$ is constructed starting from a Banach space $X$ over $\mathbb R$ equipped with a norm $\|\cdot\|_X$ via the the algebraic tensor product $X\otimes_{\mathbb R}\mathcal{F}$ of $X$ with $\mathcal{F}$.
The elements in ${\mathcal X}=X\otimes_{\mathbb R}\mathcal{F}$ are finite linear combinations of the form $\sum v_i\otimes u_i$, $v_i\in X$, where $\{u_1=1,\ldots, u_N\}$ is basis of $\mathcal{F}$. The left multiplication with $a\in\mathcal{F}$ is defined as $\sum v_i\otimes (a u_i)$ while the right multiplication by $a$ is defined by $\sum v_i\otimes ( u_ia)$. For simplicity, the symbol $\otimes$ will be omitted. The norm of $\sum v_i\otimes u_i$ in $X\otimes_{\mathbb R}\mathcal{F}$ is taken equal to
$(\sum \| v_i\|_X^2)^{1/2}$.

  We say that $\mathcal X=X\otimes_{\mathbb R}\mathcal{F}$  is a left (or right, or two-sided) module over $\mathcal{F}$ if there exist a left (or right, or two-sided) multiplication by elements of $\mathcal{F}$ and a constant $C\geq 1$ such that for all $x\in\mathcal{X}$, $a\in\mathcal{F}$ this inequality holds:
$$
\|ax\|_{\mathcal X}\leq C \|a\|_{\mathcal F}\|x\|_{\mathcal X} \qquad ({\rm or \ }\|xa\|_{\mathcal X}\leq C \|x\|_{\mathcal X} \|a\|_{\mathcal F}, {\rm \ or\ both }).
$$
  This case is useful when dealing with $n$-tuples of real linear operators.

\medskip More in general, we can consider a right (or left, or two-sided) module $\mathcal{Y}$ over $\mathcal F$.

We say that $\mathcal{Y}$ is a right (or left, or two-sided) Banach module
 over $\mathcal{F}$ if there exists a constant $C\geq 1$ such that
 $$
  \|ys\|_{\mathcal{Y}}\leq C \|y\|_{\mathcal{Y}} \|s\|_{\mathcal{F}}\ \ \
  ({\rm or}\ \ \ \ \|sy\|_{\mathcal{Y}}\leq  C\|y\|_{\mathcal{Y}}\|s\|_{\mathcal{F}},  \ \ \ {\rm or\ both})
  \ \ \ {\rm for\ all} \ \  y\in \mathcal{Y}, \ \ s\in \mathcal{F},
 $$
 and if $\mathcal Y$ is complete.
 \\
 For our purposes we will choose the norms in such a way that $C=1$.

Since we assumed $\mathbb{S}\not=\emptyset$, let $\mathrm{J}_1,\ldots ,\mathrm{J}_r$ be
a maximal set of linearly independent elements in $\mathbb{S}$ and let us consider a basis  $\{u_{\ell}\}$ of $\mathcal{F}$ such that $u_\ell=\mathrm{J}_\ell$, $\ell=1,\ldots, r$, with $r\leq N$.

We are now in s position to give the definition of vector-valued slice hyperholomorphic  functions in this more general case
which resembles the case of a quaternionic or a paravector variable.
The proofs of the main results mimic those cases closely.
We refer the reader to Section 2.3 in the book \cite{6CKG} for details.

\begin{definition}[Vector-valued slice hyperholomorphic  functions]\label{slice}

Let $\mathcal{F}$ be an algebra satisfying Assumption \ref{ASS} and denote by $\mathcal{W}_{\mathcal{F}}$ its weak cone.
 Let $U \subseteq \mathcal{W}_{\mathcal{F}}$  be an axially symmetric open set  as in \eqref{axsymm}
  where  $\mathcal{U}$ is open in $\rr^2$.

(I)
If $\mathcal{X}_L$ is a left Banach module over $\mathcal{F}$.
 A function $f:U\to \mathcal{X}_L$ is called a left-slice function, if it is of the form
 \[
 f(x) = f_{0}(u,v) + \mathrm{J}f_{1}(u,v)\qquad \text{for } x = u + \mathrm{J} v\in U
 \]
with two functions $f_{0},f_{1}: \mathcal{U} \to \mathcal{X}_L$ that satisfy the compatibility conditions
\begin{equation}\label{CCondmonVEC}
f_{0}(u,-v) = f_{0}(u,v),\qquad f_{1}(u,-v) = -f_{1}(u,v).
\end{equation}
If in addition $f_{0}$ and $f_{1}$  are of class $\mathcal C^1$ and satisfy the Cauchy-Riemann equations
 \begin{equation}\label{CRMMON}
 \begin{split}
&\frac{\partial}{\partial u} f_{0}(u,v) - \frac{\partial}{\partial v} f_{1}(u,v) = 0
\\
&
\frac{\partial}{\partial v} f_{0}(u,v)+ \frac{\partial}{\partial u} f_{1}(u,v) = 0,
\end{split}
\end{equation}
 then $f$ is called left slice hyperholomorphic.

(II)
If $\mathcal{X}_R$ is a right Banach module over $\mathcal{F}$.
Then a function $f:U\to \mathcal{X}_R$ is called a right-slice function if it is of the form
\[
f(x) = f_{0}(u,v) + f_{1}(u,v) \mathrm{J}\qquad \text{for } x = u+ \mathrm{J}v \in U
\]
with two functions $f_{0},f_{1}: \mathcal{U} \to \mathcal{X}_R$ that satisfy \eqref{CCondmonVEC}.
If in addition $f_{0}$ and $f_{1}$  are of class $\mathcal C^1$ and satisfy the Cauchy-Riemann equations (\ref{CRMMON}), then $f$ is called right slice hyperholomorphic.

(III)
If $f$ is a left (or right) slice function such that $f_{0}$ and $f_{1}$ are real-valued, then $f$ is called intrinsic.

(IV)
We denote the sets of left, right and intrinsic  slice hyperholomorphic functions on $U$ by $\mathcal{SH}_L(U)$, $\mathcal{SH}_R(U)$and $\mathcal{N}(U)$, respectively.
\end{definition}

\begin{theorem}[The Structure Formula or Representation Formula]\label{formulamon}
Let $\mathcal{F}$ be an algebra satisfying Assumption \ref{ASS} and denote by $\mathcal{W}_{\mathcal{F}}$ its weak cone.
 Let $U \subseteq \mathcal{W}_{\mathcal{F}}$  be an axially symmetric open set.

(I) Let $\mathcal{X}_L$ be a left Banach module over $\mathcal{F}$ and let $f:U\to \mathcal{X}_L$ be such that $f\in \mathcal{SH}_L(U)$.
Then, for any  vector
$x =u+\mathrm{J}_x v\in U$, the following formula holds:
\begin{equation}\label{repL}
f(x)=\frac{1}{2}\Big[1 -  \mathrm{J}\mathrm{J}_x \Big]f(u +\mathrm{J} v)+\frac{1}{2}\Big[1 +  \mathrm{J}\mathrm{J}_x \Big]f(u -\mathrm{J} v),\ \ \ {\it for \ all} \ \ \mathrm{J}\in {\mathbb S}.
\end{equation}

(II)
Let  $\mathcal{X}_R$ be a right Banach module over $\mathcal{F}$
and let $f:U\to \mathcal{X}_R$ be such that  $f\in \mathcal{SH}_R(U)$.
Then, for any  vector
$x =u+\mathrm{J}_x v\in U$, the following formula holds:
\begin{equation}\label{repR}
f(x)=\frac{1}{2}f(u +\mathrm{J} v)\Big[1 -  \mathrm{J}\mathrm{J}_x\Big]
+\frac{1}{2}f(u -\mathrm{J} v)\Big[1 +  \mathrm{J}\mathrm{J}_x\Big],\ \ \ {\it for \ all} \ \ \mathrm{J}\in{\mathbb S}.
\end{equation}
\end{theorem}
\begin{proof}
It is a direct consequence of the definition and a suitable adaptation of the proof of Proposition 2.1.19 in \cite{6CKG}.
\end{proof}
In the sequel, the symbol $S_L^{-1}(s,x)$ denotes the Cauchy kernel for left slice hyperholomorphic functions and  $S_R^{-1}(s,x)$ denotes the one for right slice hyperholomorphic functions.
\begin{definition}
Let $\mathcal{F}$ be an algebra as in Assumption \ref{ASS} and denote by $\mathcal{W}_{\mathcal{F}}$ its weak cone.
Let $x,s\in \mathcal{W}_{\mathcal{F}}$ with $x\not\in [s]$ and define:
\begin{equation}\label{LformI}
S_L^{-1}(s,x):=-(x^2 -2 \Re  (s) x+|s|^2)^{-1}(x-\overline s)
\end{equation}
\begin{equation}\label{LformII}
\ \ \ \ \ \ \ \ \ \ \ =(s-\bar x)(s^2-2\Re (x) s+|x|^2)^{-1},
\end{equation}
\begin{equation}\label{RformI}
S_R^{-1}(s,x):=-(x-\bar s)(x^2-2\Re (s)x+|s|^2)^{-1}
\end{equation}
\begin{equation}\label{RformII}
\ \ \ \ \ \ \ \ \ \ \ =(s^2-2\Re (x)s+|x|^2)^{-1}(s-\bar x).
\end{equation}
\end{definition}
The next lemma follows as in one quaternionic or Clifford algebra case, see \cite{6CKG}.
\begin{lemma} Let $x,s\in\mathcal{W}_{\mathcal{F}}$ with $s\notin [x]$.
The left slice hyperholomorphic Cauchy kernel $S_L^{-1}(s,x)$ is left slice hyperholomorphic in $x$ and right slice hyperholomorphic in $s$.
The right slice hyperholomorphic Cauchy kernel $S_R^{-1}(s,x)$ is left slice hyperholomorphic is $s$ and right slice hyperholomorphic in $x$.
\end{lemma}

\begin{definition}[Slice Cauchy domain]\label{Slice Cauchy domainSSS}
An axially symmetric open set $U \subseteq \mathcal{W}_{\mathcal{F}}$ is called a slice Cauchy domain, if $U\cap\cc_\mathrm{J}$ is a Cauchy domain in $\cc_\mathrm{J}$ for any $\mathrm{J}\in \mathbb S$. More precisely, $U$ is a slice Cauchy domain if, for any $\mathrm{J}\in {\mathbb S}$, the boundary ${\partial( U\cap\cc_\mathrm{J})}$ of $U\cap\cc_\mathrm{J}$ is the union a finite number of non-intersecting piecewise continuously differentiable Jordan curves in $\cc_{\mathrm{J}}$.
\end{definition}

\begin{theorem}[The Cauchy formula]\label{abstSCalc}
Let $\mathcal{F}$ be an algebra as in Assumption \ref{ASS} and denote by $\mathcal{W}_{\mathcal{F}}$ its weak cone.
 Let $U \subseteq \mathcal{W}_{\mathcal{F}}$  be an axially symmetric, bounded slice Cauchy domain.
 For any $\mathrm{J}\in \mathbb S$ we set $ds_\mathrm{J}=ds(-\mathrm{J})$.
Then we have:

(I)
If $\mathcal{X}_L$ is a left Banach module over $\mathcal{F}$ and let $f:U\to \mathcal{X}_L$ be a left slice hyperholomorphic function on an open set that contains $\overline{U}$.
 Then, for every $x\in U$,  we have
\begin{equation}\label{abstSCalcL}
f(x) = \frac{1}{2\pi}\int_{\partial(U\cap\mathbb{C}_\mathrm{J})}S_L^{-1}(s,x)\,ds_\mathrm{J}\,f(s),
 \ \ {\it for \ any} \ \  f\in\mathcal{SH}_L(U).
\end{equation}

(II)
If  $\mathcal{X}_R$ is a right Banach module over $\mathcal{F}$
and let $f:U\to \mathcal{X}_R$ be a right slice hyperholomorphic function on an open set that contains $\overline{U}$.
Then, for every $x\in U$,  we have
\begin{equation}\label{abstSCalcR}
f(x) = \frac{1}{2\pi}\int_{\partial(U\cap\mathbb{C}_\mathrm{J})}f(s)\,ds_\mathrm{J}\,S_R^{-1}(s,x),
 \ \ {\it for \ any} \ \  f\in\mathcal{SH}_R(U).
\end{equation}

Moreover, the  integrals (\ref{abstSCalcL}) and (\ref{abstSCalcR}) depend neither on $U$ nor on the imaginary unit $\mathrm{J}\in {\mathbb S}$.
\end{theorem}
\begin{proof}
We consider the case of functions with values in $\mathcal{X}_L$ since the other case is similar. Let us write the function $f: \ U\subseteq\mathcal{W}_{\mathcal{F}}\to\mathcal{X}_L$ as
\[
\begin{split}
f(x)&=f(u+\mathrm{J}v)=f_0(u,v)+\mathrm{J}f_1(u,v)\\
&=\sum_{i=1}^N(\varphi_{0 i}(u,v)+\mathrm{J}\varphi_{1 i}(u,v))u_i
\end{split}
\]
where, for $\ell=0,1$,
$f_i(u+\mathrm{J}v)= \varphi_{0 i}(u,v))+\mathrm{J}\varphi_{1 i}(u,v)$ is $\mathcal X_L$-valued. The Cauchy formula, computed on the complex plane $\mathbb{C}_{\mathrm{J}}$ is then valid for any $f_i$, being a function $\mathcal X_L$-valued of the complex variable $u+\mathrm{J}v$ and so it is valid for $f(u+\mathrm{J}v)$. To get $f(u+\mathrm{I}v)$ we now use the Representation formula and, finally, we use arguments as in the proof of Theorem 2.3.19 in \cite{6CKG} to show the independence of the complex plane.
\end{proof}
Similarly we have:
\begin{theorem}[Cauchy formulas on unbounded slice Cauchy domains]\label{CauchyOutsideSS}
Let $\mathcal{F}$ be an algebra satisfying Assumption \ref{ASS} and denote by $\mathcal{W}_{\mathcal{F}}$ its weak cone.
 Let $U \subseteq \mathcal{W}_{\mathcal{F}}$  be an unbounded slice Cauchy domain.
 For any $\mathrm{J}\in {\mathbb S}$ we set $ds_\mathrm{J}=ds(-\mathrm{J})$.

(I)
If $\mathcal{X}_L$ is a left Banach module over $\mathcal{F}$ and let $f:U\to \mathcal{X}_L$ be a left slice hyperholomorphic function on an open set that contains $\overline{U}$.
 If  $f(\infty) := \lim_{|x|\to\infty}f(x)$ exists, then, for every $x\in U$ we have
\[
f(x) = f(\infty) + \frac{1}{2\pi}\int_{\partial(U\cap\mathbb{C}_\mathrm{J})}S_L^{-1}(s,x)\,ds_\mathrm{J}\, f(s) \qquad\text{for any }\ \ \  f\in\mathcal{SH}_L(U).
\]

(II)
If  $\mathcal{X}_R$ is a right Banach module over $\mathcal{F}$
and let $f:U\to \mathcal{X}_R$ be a right slice hyperholomorphic function on an open set that contains $\overline{U}$.
If  $f(\infty) := \lim_{|x|\to\infty}f(x)$ exists, then, for every $x\in U$ we have
\[
f(x) = f(\infty) + \frac{1}{2\pi}\int_{\partial(U\cap\mathbb{C}_\mathrm{J})}f(s)\,ds_\mathrm{J}\,S_R^{-1}(s,x) \qquad\text{for any }\ \  f\in\mathcal{SH}_R(U).
\]
 \end{theorem}
\begin{remark}
From now on we will consider slice hyperholomorphic functions with values in a two-sided
  Banach module $\mathcal{X}$ over $\mathcal{F}$.
 \end{remark}
\begin{remark}
 With the assumption that
  the Banach module $\mathcal{X}_L$ or $\mathcal{X}_R$ is two-sided
  the sets $\mathcal{SH}_L(U)$, $\mathcal{SH}_R(U)$ become right and left modules over $\mathcal{F}$, respectively.
\end{remark}
\section{The abstract formulation of the $S$-functional calculus}

The $S$-functional calculus was originally defined for quaternionic operators and for paravector operators, i.e.,
operators of the form
$T=T_0+\sum_{j=1}^ne_jT_j$ where $T_j$ for $j=0,1,\ldots ,n$ are linear operators on a real Banach space and $\{ e_j \}_{j=1}^n$ are the units of the Clifford algebra $\mathbb{R}_n := \mathbb{R}_{0,n}$.

 The crucial fact in both cases is the notion of slice
hyperholomorphic functions defined on quaternionic numbers or on paravector numbers.

With the definition of  slice hyperholomorphic functions on a more general algebra (satisfying Assumption \ref{ASS}) and observing that the definition of $S$-spectrum is not restricted to quaternionic operators or to paravector operators
we give, after some preliminary results, the definition
of the $S$-functional calculus in its full generality.

\begin{definition}
 Let $\mathcal{Y}$ be a two-sided
  Banach module over $\mathcal{F}$ with norm $\|\cdot\|_\mathcal{Y}$.
We denote by $\mathcal{B}(\mathcal{Y})$ the left- or right-Banach module over $\mathcal F$ of all bounded right- or left-linear operators from
$\mathcal{Y}$ into itself with the supremum norm
$$
\|A\|_{\mathcal{B}(\mathcal{Y})}:=\sup_{\|v\|_\mathcal{Y}=1}\|Av\|_{\mathcal{Y}}.
$$
When no confusion can arise we will write
$\|\cdot \|$ instead of $\|\cdot\|_{\mathcal{Y}}$ (or of $\|\cdot\|_{\mathcal{B}(\mathcal{Y})}$).
\end{definition}

\begin{remark}
In this section we let $\mathcal{F}$ be an algebra satisfying Assumption \ref{ASS}
and $\mathcal{W}_{\mathcal{F}}$ its weak cone.
\end{remark}
\begin{remark} In the sequel we shall assume that $\mathcal{B}(\mathcal{Y})$ is equipped with a product of operators, with unit $\mathcal I$, and compatible with the two-sided module structure. We shall say that $\mathcal{B}(\mathcal{Y})$ is a two-sided Banach algebra over $\mathcal F$.
\end{remark}

\begin{definition}\label{abstResSeries}
Let $\mathcal{Y}$ be a two-sided Banach module over $\mathcal{F}$ and
let $A\in \mathcal{B}(\mathcal{Y})$ and $s\in \mathcal{W}_{\mathcal{F}}$. We call the series
\[
\sum_{n=0}^{+\infty} A^ns^{-n-1}\qquad\text{and}
\qquad \sum_{n=0}^{+\infty}s^{-n-1}A^n
\]
 the left and the right $S$-resolvent operator series, respectively.
\end{definition}

\begin{lemma}\label{Srespowrser}
Let $A\in \mathcal{B}(\mathcal{Y})$ and $s\in \mathcal{W}_{\mathcal{F}}$. For $\| A\| < |s|$ the left and right $S$-resolvent operator  series converge in the operator norm.
\end{lemma}

\begin{proof} For the left $S$-resolvent operator series
we have
$$
\sum_{n=0}^{+\infty} \left\|A^ns^{-n-1}\right\| \leq |s|^{-1}\sum_{n=0}^{+\infty} \left(\|A\||s|^{-1}\right)^{n}
$$
and the statement follows. We reason similarly for the right $S$-resolvent operator series.

\end{proof}

\begin{theorem}\label{LemmaQInverseSSS} Let $A\in \mathcal{B}(\mathcal{Y})$ and $s\in \mathcal{W}_{\mathcal{F}}$ with $\| A\| < |s|$. Then
\begin{equation}\label{QinverseSSS}
\left(A^2 - 2\Re(s)A + |s|^2\id\right)^{-1} = \sum_{n=0}^{+\infty}A^n \sum_{k=0}^n\bar{s}^{\, -k-1}{s}^{-n+k-1},
\end{equation}
where this series converges in the operator norm.
\end{theorem}
\begin{proof} The proof is similar to that one of operators in paravector form, so we just draw the main lines. We let
\[
a_n:=\sum_{k=0}^n\bar{s}^{\, -k-1}{s}^{-n+k-1},  \ \ \ s\in \mathcal{W}_{\mathcal{F}}
\]
be the coefficients of the operator series (\ref{QinverseSSS}) and observe that they satisfy the estimate
\[
|a_{n}|\leq \sum_{k=0}^{n}|\bar{s}|^{\, -k-1}|s|^{-n+k-1}  = (n+1)|s|^{-n-2}
\]
and so
\[
\sum_{n=0}^{+\infty}\left\| A^na_n\right\| \leq \sum_{n=0}^{+\infty}\|A\|^n|s|^{-n-2}(n+1).
\]
Since $\|A\|<|s|$, the ratio test implies the convergence of this series
and hence the series at the right hand side of \eqref{QinverseSSS} converges in the operator norm.
Moreover, we have
\begin{align*}
(A^2 - 2\Re(s)A + |s|^2\id)  \sum_{n=0}^{+\infty}A^na_n
=&\sum_{n=2}^{+\infty}A^na_{n-2} - \sum_{n=1}^{+\infty}A^na_{n-1}2\Re(s) + \sum_{n=0}^{+\infty}A^na_n|s|^2
\\
=&\sum_{n=2}^{+\infty}A^n(a_{n-2} - a_{n-1}2\Re(s) + a_n|s|^2) \\
&+ A(a_{0}2\Re(s) + a_{1}|s|^2) + \id a_{0}|s|^2.
\end{align*}
For $n\geq 2$, we have, thanks to the equalities $2\Re(s) = s + \overline{s}$ and $|s|^2 = s\overline{s} = \overline{s}s$, that
\begin{align*}
&a_{n-2} - a_{n-1}2\Re(s) + a_n|s|^2=\\
=&\sum_{k=0}^{n-2}\overline{s}^{-k-1}{s}^{-n+1+k} - \sum_{k=0}^{n-1}\overline{s}^{-k-1}2\Re(s){s}^{-n+k} + \sum_{k=0}^n\overline{s}^{-k-1}|s|^2{s}^{-n+k-1}\\
= &\sum_{k=1}^{n-1}\overline{s}^{-k}s^{-n+k} - \sum_{k=0}^{n-1} \overline{s}^{-k}s^{-n+k} - \sum_{k=0}^{n-1}\overline{s}^{-k-1}{s}^{-n+k+1} + \sum_{k=0}^{n}\overline{s}^{-k}s^{-n+k}\\
=&-s^{-n} + s^{-n} = 0.
\end{align*}
Since $s^{-1} = |s|^{-2}\overline{s}$ and $\overline{s}^{-1} = |s|^{-2}s$,
we also have that
\begin{align*}
a_02\Re(s) - a_1|s|^2 &= |s|^{-2}(s +\overline{s}) - \left(\overline{s}^{-1}s^{-2} + \overline{s}^{-2}s^{-1}\right)|s|^2
\\
&
= \overline{s}^{-1} + s^{-1} - |s|^{-2}\left(s^{-1} + \overline{s}^{-1}\right)|s|^2 = 0
\end{align*}
and so
\begin{align*}
(A^2 - 2\Re(s)A + |s|^2\id)  \sum_{n=0}^{+\infty}A^na_n= \id.
\end{align*}
Note that the coefficients $a_n$ satisfy $\overline{a_n} = a_{n}$, so they are real and hence  commute with $A$. We also observe that
\begin{align*}
  \sum_{n=0}^{+\infty}A^na_n (A^2 - 2\Re(s)A + |s|^2\id)
=(A^2 - 2\Re(s)A + |s|^2\id)  \sum_{n=0}^{+\infty}A^na_n =\id
\end{align*}
and so \eqref{QinverseSSS} holds.

\end{proof}

\begin{theorem}\label{leftrightSSS}
Let $A\in \mathcal{B}(\mathcal{Y})$ and $s\in \mathcal{W}_{\mathcal{F}}$ with $\| A\| < |s|$.

(I)
The left $S$-resolvent series equals
\[
\sum_{n=0}^{+\infty}A^ns^{-n-1} = - (A^2 - 2\Re (s)A +|s|^2\id)^{-1}(A-\overline{s}\id).
\]

(II)
The right $S$-resolvent series equals
\[
 \sum_{n=0}^{+\infty}s^{-n-1}A^n = - (A-\overline{s}\id)(A^2 - 2\Re(s)A +|s|^2\id)^{-1}.
 \]
\end{theorem}
\begin{proof} We just prove (I) as the other case (II) can be shown with a similar argument.
The result follows from the equality
\begin{equation}\label{idenimpo}
 \overline{s}\,\id-A = (A^2 - 2\Re(s)A +|s|^2\id)\sum_{n=0}^{+\infty}A^{n}s^{-1-n}.
 \end{equation}
As $A^2 - 2\Re(s)A +|s|^2\id$ is invertible by Theorem~\ref{LemmaQInverseSSS}, (\ref{idenimpo}) is equivalent to (I).
Since $s\in\mathcal{W}_{\mathcal F}$ we have that $2\Re(s) = s + \overline{s}$ and $|s|^2 = s\,\overline{s}= \overline{s}\, s  $ are real and so they commute with the operator $A$. Thus
\[
\begin{split}
(A^2-2\Re( s )A& + \vert s \vert^2\id)\sum_{n=0}^{+\infty}A^n s ^{-n-1}
\\
 = & \sum_{n=0}^{+\infty}A^{n+2} s ^{-n-1} - \sum_{n=0}^{+\infty}A^{n+1}s ^{-n-1} ( s +\overline{ s }) + \sum_{n=0}^{+\infty}A^n  s ^{-n-1}s\overline{ s }
\\
= & \sum_{n=1}^{+\infty}A^{n+1} s ^{-n} - \sum_{n=0}^{+\infty}A^{n+1} s ^{-n} -\sum_{n=0}^{+\infty}A^{n+1} s ^{-n-1} \overline{ s } + \sum_{n=0}^{+\infty}A^n s ^{-n}\overline{ s }
\\
 =&\overline{s}\,\id -A.
\end{split}
\]
\end{proof}
The previous result motivates the following definition of the $S$-spectrum in $\mathcal{W}_{\mathcal{F}}$ for operators $A\in \mathcal{B}(\mathcal{Y})$ (cfr. \cite{GR}).

\begin{definition}\label{defSspectrumSSS}
Let $A\in \mathcal{B}(\mathcal{Y})$ and $s\in \mathcal{W}_{\mathcal{F}}$, we set
\[
\Q_{s}(A):=A^2 - 2\Re(s)A + |s|^2\id.
\]
We define the $S$-resolvent set $\rho_S(A)$ of $A$ as
\[
\rho_S(A) := \{s\in \mathcal{W}_{\mathcal{F}} : \Q_{s}(A)  \text{ is invertible in $\mathcal{B}(\mathcal{Y})$}\}
\]
and we define the $S$-spectrum $\sigma_S(A)$ of $A$ as
\[
 \sigma_S(A) := \mathcal{W}_{\mathcal{F}}\setminus\rho_S(A).
\]
For $s\in\rho_{S}(A)$, the operator $\Q_{s}(A)^{-1}\in\boundOP(\mathcal{Y})$
is called its pseudo-resolvent operator  of $A$ at $s$.
\end{definition}
\begin{remark}{\rm
For $A\in \mathcal{B}(\mathcal{Y})$ an equivalent definition of the $S$-spectrum is
\[
\sigma_S(A) := \{s\in \mathcal{W}_{\mathcal{F}}: \Q_{s}(A)  \text{ is not invertible in $\mathcal{B}(\mathcal{Y})$}\}
\]
and the  $S$-resolvent set $\rho_S(A)$ is defined as
\[
 \rho_S(A) := \mathcal{W}_{\mathcal{F}}\setminus\sigma_S(A).
\]
Definition \ref{defSspectrumSSS} of the $S$-resolvent set and of the $S$-spectrum is more useful for unbounded operators.
}
\end{remark}

As the following result shows, also in this more general case the $S$-spectrum is axially symmetric and so it is compatible with the structure of slice hyperholomorphic functions.

\begin{proposition}\label{AxSymSpec}
Let $A\in \mathcal{B}(\mathcal{Y})$. The sets $\rho_{S}(A)$ and $\sigma_S(A)$ are axially symmetric in $\mathcal{W}_{\mathcal{F}}$.
\end{proposition}
\begin{proof}
If $s = u+\mathrm{J}v \in \mathcal{W}_{\mathcal{F}}$ and  $\tilde{s} = u + \mathrm{I} v$ is any other element in $[s]$, then
\[
\Q_{\tilde{s}}(A) =A^2 - 2u A + (u^2+v^2)\id = \Q_{s}(A).
\]
It is clear that $\Q_{\tilde{s}}(A)$ is invertible if and only if $\Q_{s}(A)$ is invertible and so $s\in\rho_{S}(A)$ if and only if $\tilde{s}\in\rho_{S}(A)$. Therefore $\rho_{S}(A)$ and $\sigma_{S}(A)$ are axially symmetric.
\end{proof}

\begin{definition}[$S$-resolvent operators]\label{ResolventRegularAAA}
Let $A\in \mathcal{B}(\mathcal{Y})$. For $s\in\rho_S(A)$, we define the {\em left $S$-resolvent operator} as
$$
 S_L^{-1}(s,A) = -\Q_{s}(A)^{-1}(A-\overline{s}\,\id),
 $$
and the {\em right $S$-resolvent operator} as
$$
S_R^{-1}(s,A) = -(A-\overline{s}\id)\Q_{s}(A)^{-1}.
$$
\end{definition}
\begin{lemma}
\label{ResolventRegularA}
Let $A\in \mathcal{B}(\mathcal{Y})$.

(I) The left $S$-resolvent $S_L^{-1}(s,A)$ is a $ \mathcal{B}(\mathcal{Y})$-valued right-slice hyperholomorphic function of the variable $s$ on $\rho_S(A)$.

(II)  The right $S$-resolvent  $S_R^{-1}(s,A)$ is a $\mathcal{B}(\mathcal{Y})$-valued left-slice hyperholomorphic
     function of the variable $s$ on $\rho_S(A)$.

\end{lemma}
\begin{proof}
We prove only (I), since the proof of (II) is similar. Let $s = u+ \mathrm{J}v \in\rho_S(A)$, then
\[
S_L^{-1}(s,A)= f_0(u,v) + f_1(u,v)\mathrm{J}
\]
with the $\mathcal{B}(\mathcal{Y})$-valued functions
\begin{align*}
f_0(u,v) &= -(A^2-2uA + (u^2+v^2)\id)^{-1}(A-u\id),\\
f_1(u,v) &= - (A^2-2uA+ (u^2+v^2)\id)^{-1}v.
\end{align*}
The functions $f_0$ and $f_1$ satisfy the compatibility condition \eqref{CCondmonVEC}, the function
$S_L^{-1}(s,A)$ is a $\mathcal{B}(\mathcal{Y})$-valued right slice function on $\rho_S(A)$.
We verify that the pair $(f_0,f_1)$ satisfies the Cauchy-Riemann equations \eqref{CRMMON}. We have:
\[
\begin{split}
\frac{\partial}{\partial u}(-Q_s(A)^{-1}(A-u\id))&= (Q_s(A)^{-2}(-2A +2u \id)(A-u\id))+Q_s(A)^{-1}\\
&=  Q_s(A)^{-2}[-2(A^2 - 2u A +u^2 \id)+ A^2 - 2u A +(u^2+v^2) \id]\\
&=  Q_s(A)^{-2}(-A^2 + 2u A +(-u^2+v^2)\id),
\end{split}
\]
\[
\begin{split}
\frac{\partial}{\partial v}(-Q_s(A)^{-1}v)&= Q_s(A)^{-2} (2v^2) - Q_s(A)^{-1}\\
&= Q_s(A)^{-2} (- A^2 +2 uA + (-u^2+v^2)\id)
\end{split}
\]
\[
\begin{split}
\frac{\partial}{\partial v}(-Q_s(A)^{-1}(A-u\id))&= Q_s(A)^{-2}(2v \id)(A-u\id)
\end{split}
\]
\[
\begin{split}
\frac{\partial}{\partial u}(-Q_s(A)^{-1}v)&= Q_s(A)^{-2} (-2 A+2u \id)v
\end{split}
\]
and the statement immediately follows.
\end{proof}

The following result on the invertibility of operators follows with standard arguments, so we provide a reference for the proof.
\begin{lemma}\label{inverso}
The set ${\rm Inv}(\mathcal{B}(\mathcal{Y}))$ of invertible elements in $\mathcal{B}(\mathcal{Y})$
 is an open set in the uniform operator topology on $\mathcal{B}(\mathcal{Y})$. If ${\rm Inv}(\mathcal{B}(\mathcal{Y}))$ contains an element
$A$, then it contains the ball
$$
B_{\delta}(A) := \left\{ C \in \mathcal{B}(\mathcal{Y})\ :\ \|A-C\| < \delta \right\},
$$
where $\delta = \left\| A^{-1} \right\|^{-1}$. If $C\in B(A)$, then inverse is given by the series
\begin{equation}\label{Bmenouno}
C^{-1}=A^{-1}\sum_{m = 0}^{+\infty}[(A-C)A^{-1}]^m.
\end{equation}
Furthermore, the map $A\mapsto A^{-1}$ is a homeomorphism
from ${\rm Inv}(\mathcal{B}(\mathcal{Y}))$ onto ${\rm Inv}(\mathcal{B}(\mathcal{Y}))$  in the uniform operator topology.
\end{lemma}
\begin{proof}
See Lemma 3.1.12 in \cite{6CKG}.
\end{proof}
The $S$-spectrum in $\mathcal{W}_{\mathcal F}$ has the standard properties of the spectrum of a complex linear operator in particular, it is compact:

\begin{theorem}[Compactness of the $S$-spectrum]\label{PropSpec}
Let $A\in \mathcal{B}(\mathcal{Y})$.
The $S$-spectrum $\sigma_S(A)$ of $A$ is a nonempty and compact set contained in the closed ball  $\overline{B_{\|A\|}(0)}$ with center at $0$ and radius $\|A\|$.
\end{theorem}
\begin{proof}
For $r$ such that $r>\|A\|$, the series  $S_L^{-1}(s,A) = \sum_{m=0}^{+\infty}A^ms^{-m-1}$ converges uniformly on the boundary $\partial B_r(0)$ of the ball $B_r(0)$ centered at $0$ with radius $r$. Thus, if $\mathrm{J}\in \mathbb{S}$, we have
\begin{align}
\label{TABC}  \int_{\partial(B_r(0)\cap\mathbb{C}_\mathrm{J})} S_L^{-1}(s,A) \,ds_\mathrm{J} = \sum_{m=0}^{+\infty}A^m \int_{\partial(B_r(0)\cap\mathbb{C}_\mathrm{J})}s^{-m-1}\,ds_\mathrm{J} = 2\pi\,\id,
 \end{align}
since the integral $ \int_{\partial(B_r(0)\cap\mathbb{C}_\mathrm{J})}s^{-m-1}ds_\mathrm{J}=2\pi$  if $m=0$ while it vanishes for $m\in\mathbb N$.
If we assume, by absurd, that $\overline{B_{r}(0)}\subseteq\rho_{S}(T)$,
then $S_{L}^{-1}(s,A)$ is right slice hyperholomorphic on $\overline{B_{r}(0)}$ by Lemma~\ref{ResolventRegularA} and by
 Cauchy's integral theorem the integral in \eqref{TABC} vanishes, which is a contradiction.
   Thus  $\overline{B_{r}(0)} \not\subset \rho_{S}(A)$ and so $\emptyset \neq \sigma_{S}(A)\cap\overline{B_{r}(0)}$ and $\sigma_{S}(A)$ is not empty.

The set $\mathcal{B}(\mathcal{Y})$ can be considered as a real Banach algebra, by restricting the multiplication by a scalar to $\mathbb{R}$. The set ${\rm Inv}(\mathcal{B}(\mathcal{Y}))$ of  invertible elements of this real Banach algebra is open by Lemma \ref{inverso}. The map
$$
\tau:s\mapsto \Q_{s}(A)$$
 is a $\mathcal{B}(\mathcal{Y})$-valued continuous function so that the set
 $$
 \rho_S(A) = \tau^{-1}({\rm Inv}(\mathcal{B}(\mathcal{Y})))
 $$
  is open in $\mathcal{W}_{\mathcal{F}}$. We conclude that  $\sigma_S(A)$ is  closed.

Finally, Lemma \ref{LemmaQInverseSSS} implies $|s| \leq \|A\|$ for any  $s\in\sigma_S(A)$. Thus, $\sigma_S(A)$ is a closed subset of the compact set $\overline{B_{\|A\|}(0)}$ and therefore compact itself.

\end{proof}
Each of the two $S$-resolvents satisfy a suitable equation as shown in the next result.
\begin{theorem}\label{312}
Let $A\in\mathcal{B}(\mathcal{Y})$ and let $s\in\rho_S(A)$. The left $S$-resolvent operator satisfies the {\em left $S$-resolvent equation}

\begin{equation}\label{LeftSREQSS}
S_L^{-1}(s,A)s - AS_L^{-1}(s,A) = \id
\end{equation}
and the right $S$-resolvent operator satisfies the {\em right $S$-resolvent equation}
\begin{equation}\label{RightSREQ}
sS_R^{-1}(s,A) - S_R^{-1}(s,A)A = \id.
\end{equation}
\end{theorem}
\begin{proof}
It is a simple computation, as in the case of paravector operators.
\end{proof}

The left and the right $S$-resolvent equations cannot be considered the generalizations of the classical resolvent equation.
The $S$-resolvent equation involves both the $S$-resolvent operators and the Cauchy kernels.
\begin{theorem}[The $S$-resolvent equation]\label{SREQSSS}
Let $A\in\mathcal{B}(\mathcal{Y})$ and let $s, q\in\rho_S(A)$ with $q\notin[s]$.
 Then the equation
\begin{multline}\label{SREQ1SSS}
S_R^{-1}(s,A)S_L^{-1}(q,A)=\left[\left(S_R^{-1}(s,A)-S_L^{-1}(q,A)\right)q
-\overline{s}\left(S_R^{-1}(s,A)-S_L^{-1}(q,A)\right)\right]\Q_{s}(q)^{-1}
\end{multline}
holds true. Equivalently, it can also be written as
\begin{multline}\label{SREQ2SSS}
S_R^{-1}(s,A)S_L^{-1}(q,A)=\Q_{q}(s)^{-1}
\left[\left(S_L^{-1}(q,A) - S_R^{-1}(s,A)\right)\overline{q}-s\left(S_L^{-1}(q,A) - S_R^{-1}(s,A)\right)
 \right],
\end{multline}
where we have set $\Q_{s}(q)^{-1}:=(q^2-2\Re(s)q+|s|^2)^{-1}$.
\end{theorem}
\begin{proof}
The proof follows in much the same manner as the proof in the case of quaternionic or paravector operators.
Indeed, with some manipulations, using  the left and the right $S$-resolvent (\ref{LeftSREQSS}),
we get
\begin{align*}
S_R^{-1}(s,A)S_L^{-1}(q,A)(q^2-2s_0q+|s|^2)
=&(s^2-2s_0s+|s|^2)S_R^{-1}(s,A)S_L^{-1}(q,A)
\\
&+[S_R^{-1}(s,A)-S_L^{-1}(q,A)]q-\overline{s}[S_R^{-1}(s,A)-S_L^{-1}(q,A)]
\end{align*}
and since $s^2-2s_0s+|s|^2=0$, thanks to Lemma \ref{squadequat} we obtain \eqref{SREQ1SSS}.
With similar computations we can show that also \eqref{SREQ2SSS} holds.
For more details see the proof of Theorem 3.1.15 in \cite{6CKG}.
\end{proof}

We now define the $S$-functional calculus starting with the case of polynomials of $A$.

\begin{lemma}\label{PolyTABSSS}
Let $A\in\boundOP(\mathcal{Y})$, let $m\in  \mathbb{N}\cup\{0\}$ and let $U \subseteq \mathcal{W}_{\mathcal{F}}$
be a bounded slice Cauchy domain with $\sigma_{S}(A)\subset U$. For any imaginary unit $\mathrm{J}\in\mathbb{S}$ we set $ds_\mathrm{J}=ds(-\mathrm{J})$. Then we have
\begin{align*}
A^m =& \frac{1}{2\pi}\int_{\partial(U\cap\mathbb{C}_\mathrm{J})}S_L^{-1}(s,A)\,ds_\mathrm{J}\,s^m
\end{align*}
and also
\begin{align*}
A^m =& \frac{1}{2\pi}\int_{\partial(U\cap\mathbb{C}_\mathrm{J})}s^m\,ds_\mathrm{J}\,S_R^{-1}(s,A).
 \end{align*}
\end{lemma}
\begin{proof}
In the case  $U$ is the ball $B_r(0)$ in $\mathcal{W}_{\mathcal{F}}$ with center at $0$ and radius $r>\|A\|$, then $S_L^{-1}(s,A) = \sum_{n=0}^{+\infty}A^ns^{-n-1}$ for any
$s\in\partial B_r(0)$ by Theorem~\ref{leftrightSSS} and the series converges uniformly on $\partial B_r(0)$. So we conclude that
\begin{align*}
&\frac{1}{2\pi}\int_{\partial(B_r(0)\cap\mathbb{C}_\mathrm{J})}S_L^{-1}(s,A)\, ds_\mathrm{J}\, s^m = A^m.
\end{align*}
In the case when $U$ is an arbitrary bounded slice Cauchy domain containing $\sigma_{S}(A)$, there exists $r>0$ such that the ball  $B_r(0)$ contains $\overline{U}$. The left $S$-resolvent $S_L^{-1}(s,A)$ is then right slice hyperholomorphic and the monomial $s^m$ is  left slice hyperholomorphic on the bounded slice Cauchy domain $B_{r}(0)\setminus U$. Cauchy's integral theorem implies
\begin{align*} &\frac{1}{2\pi}\int_{\partial(B_r(0)\cap\mathbb{C}_\mathrm{J})}S_L^{-1}(s,A)\,ds_\mathrm{J}\,s^m
- \frac{1}{2\pi}\int_{\partial(U\cap\mathbb{C}_\mathrm{J})}S_L^{-1}(s,A)\,ds_\mathrm{J}\,s^m
\\
=& \frac{1}{2\pi}\int_{\partial((B_r(0)\setminus U)\cap\mathbb{C}_\mathrm{J})}S_L^{-1}(s,A)\,ds_\mathrm{J}\,s^m = 0
\end{align*}
and so
\[
\frac{1}{2\pi}\int_{\partial(U\cap\mathbb{C}_\mathrm{J})}S_L^{-1}(s,A)\,ds_\mathrm{J}\,s^m =
\frac{1}{2\pi}\int_{\partial(B_r(0)\cap\mathbb{C}_\mathrm{J})}S_L^{-1}(s,A)\,ds_\mathrm{J}\,s^m = A^m.
\]

The second formula in the statement, involving $S_R^{-1}(s,A)$ follows similarly.

\end{proof}

\begin{theorem}\label{TGenPolySSS}
 Let $A\in\boundOP(\mathcal{Y})$, let $U$ be a bounded slice Cauchy domain that contains $\sigma_{S}(A)$ and let $\mathrm{J}\in\mathbb{S}$. For any left slice hyperholomorphic polynomial
 $P(x) = \sum_{\ell=0}^M x^{\ell}a_{\ell}$ with $a_{\ell}\in\mathcal{F}$,
 we set
 $
 P(A) = \sum_{\ell=0}^MA^{\ell}a_{\ell}.
 $
  Then
\begin{equation}\label{GenPolyL}
P(A) = \frac{1}{2\pi}\int_{\partial(U\cap\mathbb{C}_\mathrm{J})} S_L^{-1}(s,A)\,ds_\mathrm{J}\, P(s).
\end{equation}
 Similarly, we set $P(A) = \sum_{\ell=0}^Ma_{\ell}A^{\ell}$ for any right slice hyperholomorphic polynomial
 $P(x) = \sum_{\ell=0}^M a_{\ell}x^{\ell}$ with $a_{\ell}\in\mathcal{F}$.
 Then
\begin{equation}
\label{GenPolyR} P(A) = \frac{1}{2\pi}\int_{\partial(U\cap\mathbb{C}_\mathrm{J})}P(s)\,ds_\mathrm{J}\, S_R^{-1}(s,A).
\end{equation}
In particular, the operators in \eqref{GenPolyL} and \eqref{GenPolyR}
coincide for any intrinsic polynomial
$P(x) = \sum_{\ell=0}^{M}x^{\ell}a_{\ell}$, i.e., when the coefficients $a_{\ell}\in\mathbb{R}$.
\end{theorem}
\begin{proof}
It is a direct consequence of  Lemma \ref{PolyTABSSS}.
\end{proof}

The $S$-functional calculus is based, more in general, on functions that are slice hyperholomorphic
 on the $S$-spectrum of an operator $A$.
To define the calculus in its full generality we need some more notations.
\begin{definition}
Let $A\in\boundOP(\mathcal{Y})$.
We denote by $\mathcal{SH}_L(\sigma_{S}(A))$,
$\mathcal{SH}_R(\sigma_{S}(A))$
and $\mathcal{N}(\sigma_{S}(A))$,
the set of all  left, right and intrinsic  slice hyperholomorphic  functions
with $\sigma_{S}(A)\subset \dom(f)$  and $\dom(f)$ is the domain of the function $f$.
\end{definition}

\begin{definition}[The abstract formulation $S$-functional calculus]\label{SCalc}
Let $A\in\boundOP(\mathcal{Y})$. For any imaginary unit $\mathrm{J}\in\mathbb{S}$ we set $ds_\mathrm{J}=ds(-\mathrm{J})$ and we define
\begin{equation}\label{abstSCalcLSSS}
f(A) := \frac{1}{2\pi}\int_{\partial(U\cap\mathbb{C}_\mathrm{J})}S_L^{-1}(s,A)\,ds_\mathrm{J}\,f(s),
 \ \ {\rm for \ any} \ \  f\in\mathcal{SH}_L(\sigma_{S}(A)),
\end{equation}
and
\begin{equation}\label{abstSCalcRSSS}
f(A) := \frac{1}{2\pi}\int_{\partial(U\cap\mathbb{C}_\mathrm{J})}f(s)\,ds_\mathrm{J}\,S_R^{-1}(s,A),
 \ \ {\rm for \ any} \ \  f\in\mathcal{SH}_R(\sigma_{S}(A)),
 \end{equation}
where $U\subset\mathcal{W}_{\mathcal{F}}$ is a bounded Cauchy domain such that
$\sigma_{S}(A)\subset U$ and $\overline{U}\subset\dom(f)$.
\end{definition}

\begin{definition}[Universality property]
The {\em universality property of
the $S$-functional calculus} is the formulation of the $S$-functional calculus
with respect to the algebra of operators
$\mathcal{B}(\mathcal{Y})$
and for slice hyperholomorphic functions associated with an algebra satisfying Assumption \ref{ASS}.
\end{definition}

The abstract formulation of the $S$-functional calculus is well-defined because
the integrals (\ref{abstSCalcLSSS}) and (\ref{abstSCalcRSSS}) do not depend on the open set $U$ that contain the $S$-spectrum and
 on  $\mathrm{J}\in\mathbb{S}$.

\begin{theorem}\label{OperatorIntsSSS}
Let $A\in\boundOP(\mathcal{Y})$. For any $f\in\mathcal{SH}_L(\sigma_{S}(A))$, the integral in \eqref{abstSCalcLSSS} defining the operator $f(A)$ is independent of the choice of the slice Cauchy domain $U$ and of the imaginary unit $\mathrm{J}\in\mathbb{S}$. Similarly, for any  $f\in\mathcal{SH}_R(\sigma_{S}(A))$, the integral in \eqref{abstSCalcRSSS} that defines the operator $f(A)$ is also independent of the choice of $U$ and $\mathrm{J}\in\mathbb{S}$.
\end{theorem}

\begin{proof}
We follow the proof in the quaternionic case, and we first show the independence of the definition of the choice of the slice Cauchy domain $U$.
So we choose another bounded slice Cauchy domain $U'\subset\mathcal{W}_{\mathcal{F}}$ containing $\sigma_{S}(A)$ and such that its closure is contained in the domain of $f\in\mathcal{SH}_L(\sigma_{S}(A))$.
If $\overline{U'}\not\subset U$, then $O:= U \cap U'$ is an axially symmetric open set, $O\supset\sigma_{S}(A)$.
We can hence find another slice Cauchy domain $U''$ with $\sigma_{S}(A)\subset U''$ and $\overline{U''} \subset O = U \cap U'$. The integrals over the boundaries of all three sets $U, U',U''$ agree because $f$ is slice hyperholomorphic. On the other hand, if $\overline{U'}\subset U$ we set $O=U\setminus U'$ and we observe that $\overline{O}$ is a bounded Cauchy domain whose closure is contained in $\rho_S(T)$. Thus the integral \eqref{abstSCalcLSSS} computed over $\partial (O\cap\mathbb C_{\mathrm J})$ vanishes, which implies
$$
\frac{1}{2\pi}\int_{\partial(U\cap\mathbb{C}_\mathrm{J})}S_L^{-1}(s,A)\,ds_\mathrm{J}\,f(s)-
\frac{1}{2\pi}\int_{\partial(U'\cap\mathbb{C}_\mathrm{J})}S_L^{-1}(s,A)\,ds_\mathrm{J}\,f(s)=0
$$
and the assertion follows also in this case.

We now  show the independence of the imaginary unit: let $\mathrm{I}$, $\mathrm{J}\in\mathbb{S}$ and let $U_q,U_s\subset \dom(f)$ with $\sigma_{S}(A)\subset U_q$ be  two slice Cauchy domains and assume $\overline{U_q}\subset U_s$. (The subscripts $q$ and $s$ denote the respective variable of integration in the following computation). The set $U_q^c := \mathcal{W}_{\mathcal{F}} \setminus U_q$  is an unbounded axially symmetric slice Cauchy domain with $\overline{U_q^c}\subset \rho_{S}(A)$. The $S$-resolvent $S_L^{-1}(s,A)$ is right slice hyperholomorphic on $\rho_{S}(A)$ and also at infinity so the right slice hyperholomorphic Cauchy formula implies
\begin{align*}
S_L^{-1}(s,A) =& \frac{1}{2\pi}\int_{\partial (U_q^c\cap\cc_\mathrm{I})} S_L^{-1}(q,A)\,dq_{\mathrm{I}}\,S_R^{-1}(q,s)
\end{align*}
for any $s\in U_s$. Since $\partial(U_q^c\cap\cc_\mathrm{J}) = - \partial (U_q\cap\cc_\mathrm{J})$ and  $S_R^{-1}(q,s) = - S_L^{-1}(s,q)$, we get
\begin{align*}
f(A)
=& \frac{1}{2\pi}\int_{\partial(U_s\cap\cc_\mathrm{J})} S_L^{-1}(s,A)\,ds_\mathrm{J}\,f(s)
\\
=
&  \frac{1}{(2\pi)^2}\int_{\partial(U_s\cap\cc_\mathrm{J})} \left(\int_{\partial (U_q^c\cap\cc_\mathrm{I})} S_L^{-1}(q,A)\,dq_{\mathrm{I}}\,S_R^{-1}(q,s)\right)\,ds_\mathrm{J}\,f(s)
\\
=
& \frac{1}{(2\pi)^2} \int_{\partial (U_q\cap\cc_\mathrm{I})} S_L^{-1}(q,A)\,dq_{\mathrm{I}}\left(\int_{\partial(U_s\cap\cc_\mathrm{J})}S_L^{-1}(s,q)\,
ds_\mathrm{J}\,f(s)\right)
\\
=
& \frac{1}{2\pi} \int_{\partial (U_q\cap\cc_\mathrm{I})} S_L^{-1}(q,A)\,dq_{\mathrm{I}}f(q),
\end{align*}
and the statement follows.
\end{proof}

We note that if $f\in \mathcal{N}(\sigma_S(A))$ then formulas \eqref{abstSCalcLSSS} and \eqref{abstSCalcRSSS} give the same operator. This can be shown as for paravector operators
 by uniform approximation of $f$ with intrinsic rational functions.
 This result and the fact that the $S$-functional calculus is consistent with the limits of uniformly convergent
 sequences of slice hyperholomorphic functions can be shown as in the case of paravector operators.

\begin{theorem}\label{SCalcBimonSSS}
 Let $A\in\boundOP(\mathcal{Y})$. If $f\in\mathcal{N}(\sigma_S(A))$,
 then both versions of $S$-functional calculus give the same operator $f(A)$, i.e., we have
\begin{equation*}\label{SCalcBimonEQ}
f(A)=\frac{1}{2\pi}\int_{\partial(U\cap\mathbb{C}_\mathrm{J})}S_L^{-1}(s,A)\, ds_\mathrm{J}\, f(s)= \frac{1}{2\pi}\int_{\partial(U\cap\mathbb{C}_\mathrm{J})}f(s)\,ds_\mathrm{J}\,S_R^{-1}(s,A).
\end{equation*}
\end{theorem}
\begin{proof}
It is easy to check that the proof follows from Runge's theorem (see Theorem 2.1.37 in \cite{6CKG}) adapted to this general setting and it follows the same lines as when $A$ is a paravector operator.

\end{proof}

\section{An application to bounded full Clifford operators}

As we mentioned the $S$-functional calculus was introduced for quaternionic operators and for
$(n+1)$-tuples $(T_0,...,T_n)$ of linear operators written in paravector form as
$T=T_0+\sum_{j=1}^ne_jT_j$.
The universality property of the $S$-functional calculus highlights the great generality of this calculus and here we show how it can be applied not only to paravector operators with noncommuting components but also to full Clifford operators with noncommuting components.
Such operators contain as a particular case paravector operators.
The main references in which the formulations and the properties of $S$-functional calculus
for quaternionic operators and for paravector operators are \cite{acgs,CSSFUNCANAL}.

\begin{remark}
The great adaptability of the $S$-functional calculus to several settings, formalized in terms of its universality property, is not
shared by other functional calculi based on different Cauchy formulas, but it is a consequence of the
Cauchy formula with slice hyperholomorphic Cauchy kernels. This fact has several advantages not only in operator theory but also for the function theory.
\end{remark}

To give an application of the $S$-functional calculus to full Clifford operators
we need slice hyperholomorphic functions with values in a Clifford algebra (slice monogenic functions).
Let $\rr_n$ be the real Clifford algebra over $n$ imaginary units $e_1,\ldots ,e_n$
satisfying the relations $e_\ell e_m+e_me_\ell=0$,\  $\ell\not= m$, $e_\ell^2=-1.$
 An element in the Clifford algebra will be denoted by $\sum_A e_Ax_A$ where
$A=(\ell_1\ldots \ell_r)$, $\ell_i\in\{1,2,\ldots, n\},\ \  \ell_1<\ldots <\ell_r$
 is a multi-index
and $e_A=e_{\ell_1} e_{\ell_2}\ldots e_{\ell_r}$, $e_\emptyset =1$.
An element $(x_0,x_1,\ldots,x_n)\in \mathbb{R}^{n+1}$  will be identified with the element
$
 x=x_0+\underline{x}=x_0+ \sum_{\ell=1}^nx_\ell e_\ell\in\mathbb{R}_n
$
called paravector and the real part $x_0$ of $x$ will also be denoted by $\Re(x)$.
The norm of $x\in\mathbb{R}^{n+1}$ is defined as $|x|^2=x_0^2+x_1^2+\ldots +x_n^2$.
 The conjugate of $x$ is defined by
$
\bar x=x_0-\underline x=x_0- \sum_{\ell=1}^nx_\ell e_\ell.
$
We denote by $\mathbb{S}$ the sphere
$$
\mathbb{S}=\{ \underline{x}=e_1x_1+\ldots +e_nx_n\ :  \  x_1^2+\ldots +x_n^2=1\};
$$
for $j\in\mathbb{S}$ we obviously have $j^2=-1$.
Given an element $x=x_0+\underline{x}\in\rr^{n+1}$ let us set
$
j_x=\underline{x}/|\underline{x}|$ if $\underline{x}\not=0,
$
 and given an element $x\in\rr^{n+1}$, the set
$$
[x]:=\{y\in\rr^{n+1}\ :\ y=x_0+j |\underline{x}|, \ j\in \mathbb{S}\}
$$
is an $(n-1)$-dimensional sphere in $\mathbb{R}^{n+1}$.
The vector space $\mathbb{R}+j\mathbb{R}$ passing through $1$ and
$j\in \mathbb{S}$ will be denoted by $\mathbb{C}_j$ and
an element belonging to $\mathbb{C}_j$ will be indicated by $u+jv$, for $u$, $v\in \mathbb{R}$.
With an abuse of notation we will write $x\in\mathbb{R}^{n+1}$.
Thus, if $U\subseteq\mathbb{R}^{n+1}$
a function $f:\ U\subseteq \mathbb{R}^{n+1}\to\mathbb{R}_n$ can be interpreted as
a function of the paravector $x$.
With the above notation the definition of the slice hyperholomorphic functions $f:\ U\subseteq \mathbb{R}^{n+1}\to\mathbb{R}_n$
is analogous to the notion of slice hyperholomorphic functions for Clifford algebra-valued functions.
Precisely we use the Definition \ref{slice}
for Clifford algebra valued-functions, i.e., slice monogenic functions.

\begin{remark}
Assumption \ref{ASS} on the algebra $\mathcal{F}$ and the fact that the functions are defined on  the weak cone allow us to consider a large set of possibilities
to define the $S$-functional calculus. We can also choose a subset of the weak cone in order to define the $S$-spectrum. For example, we do not need to take the full weak cone
$\mathcal{W}_{\mathcal{F}}$ when dealing with a Clifford algebra $\mathbb R_n$.
For our purposes, when we are working with $\mathbb{R}_n$-valued functions, it is enough to work with a subset of the weak cone, namely, the paravectors.
\end{remark}

So with the definitions of Section \ref{UNIVERS} we make the following identifications in order to apply the abstract formulation of the $S$-functional calculus.
\begin{itemize}
\item[(I)] $\mathcal{F}=\mathbb{R}_n$,

\item[(II)] We consider the sphere $\mathbb{S}=\{ \underline{x}=e_1x_1+\ldots +e_nx_n\ :  \  x_1^2+\ldots +x_n^2=1\}$,

\item[(III)] In $\mathcal{W}_\mathcal{F}$ we pick the paravectors identified with $\mathbb{R}^{n+1}=\bigcup_{j\in \mathbb{S}}\mathbb{C}_j$,

\item[(IV)] The involution is $s=s_0+\underline{s}\mapsto \overline{s}=s_0-\underline{s}$, for  all $s\in \mathbb{R}^{n+1}$.
\end{itemize}

In this case
$s_0={\rm Re}(s)=\frac{1}{2}(s+\overline{s})$ and $s\overline{s}=\overline{s}s=s_0^2+...+s_n^2$.
The definition of an {\em axially symmetric} set is as in the Clifford  setting, i.e.,
 we say that $U \subseteq \mathbb{R}^{n+1}$ is axially symmetric if $[x]\subset U$  for any $x \in U$.

 \medskip
 Then we consider the functional setting for operators.
We will consider a real Banach space $V$ over
$\mathbb{R}$
 with norm $\|\cdot \|$.
 By $V_n$ we denote $V\otimes \rr_n$ over $\rr_n$; $V_n$ can be made into a
 two-sided Banach module as explained in Section 2.
 An element in $V_n$ is of the type $\sum_A v_A\otimes e_A$.

We denote by
$\mathcal{B}(V)$  the space
of bounded $\mathbb{R}$-homomorphisms of the Banach space $V$ to itself
 endowed with the natural norm denoted by $\|\cdot\|_{\mathcal{B}(V)}$.
Given $T_A\in \mathcal{B}(V)$, we can introduce the full Clifford operator
\begin{equation}\label{fullyCliop}
\hat{T}=\sum_A T_Ae_A
\end{equation}
 and
its action on $v=\sum_B v_Be_B\in V_n$ as
$$
\hat{T}(v)=\sum_{A,B}T_A(v_B)e_Ae_B.
$$
 The operator $\hat{T}$ is a right-module homomorphism which is a bounded linear
map on $V_n$ and the norm is given by
\begin{equation}\label{FULLYCLIF}
\|\hat{T}\|_{\mathcal{B}(V_n)}:=\sum_A \|T_A\|_{\mathcal{B}(V)}.
\end{equation}
The paravector operators  are particular Clifford operators of the form
$T=T_0+\sum_{j=1}^ne_jT_j$ where $T_j\in\mathcal{B}(V)$ for $j=0,1,\ldots ,n$.
So to avoid confusion with the previous literature
a (full) Clifford operator is denoted by $\hat{T}$ in order to distinguish it
from paravectors operators usually denoted by $T$.
The subset of paravector operators in ${\mathcal{B}(V_n)}$ are usually denoted by $\mathcal{B}^{\small 0,1}(V_n)$.
Observe that we have $\|T\|_{\mathcal{B}^{\small 0,1}(V_n)}=\sum_j \|T_j\|_{\mathcal{B}(V)}$.
It is evident from \eqref{fullyCliop} and \eqref{FULLYCLIF} that if $\mathcal{B}(V)$ is real Banach algebra, $\mathcal{B}(V_n)$ is a Banach algebra over $\mathbb R_n$. In the sequel we will omit the subscript $\mathcal{B}(V_n)$ in the norm of an operator.
Note also that  $\|\hat{T}\hat{S}\|\leq \|\hat{T}\| \|\hat{S}\|$.

\medskip
To apply the $S$-functional calculus to full Clifford operators we choose
$$
\mathcal{Y}=V_n=V\otimes \rr_n.
$$
So if we let
 $\hat{T}\in\mathcal{B}(V_n)$ and $s\in \mathbb{R}^{n+1}$ the series
\[
\sum_{m=0}^{+\infty} \hat{T}^ms^{-m-1}\qquad\text{and}
\qquad \sum_{m=0}^{+\infty}s^{-m-1}\hat{T}^m
\]
are called the left and the right $S$-resolvent operator series associated with the Clifford operator $\hat{T}$, respectively.
The following facts are now a direct consequence of the abstract formulation
of the $S$-functional calculus.

\begin{theorem}\label{LemmaQInverse}
Let $\hat{T}\in\mathcal{B}(V_n)$ and let $s\in \mathbb{R}^{n+1}$ with $\|\hat{T}\|<|s|$.
Then we have:

(I) The left and right $S$-resolvent operator  series converge in the operator norm.

(II) We have
\begin{equation}\label{Qinverse}
\left(\hat{T}^2 - 2\Re(s)\hat{T} + |s|^2\id\right)^{-1} = \sum_{m=0}^{+\infty}\hat{T}^m \sum_{k=0}^n(\overline{s})^{-k-1}({s})^{-m+k-1},
\end{equation}
where this series converges in the operator norm.

(III)
The left $S$-resolvent series equals
\begin{equation}\label{CSSL}
\sum_{m=0}^{+\infty}\hat{T}^ms^{-m-1} = - (\hat{T}^2 - 2\Re (s)\hat{T} +|s|^2\id)^{-1}(\hat{T}-\overline{s}\id).
\end{equation}

(IV)
The right $S$-resolvent series equals
\begin{equation}\label{CSSR}
 \sum_{m=0}^{+\infty}s^{-m-1}\hat{T}^m = - (\hat{T}-\overline{s}\id)(\hat{T}^2 - 2\Re(s)\hat{T} +|s|^2\id)^{-1}.
 \end{equation}
\end{theorem}
The definition of the $S$-spectrum and of the $S$-resolvent set in this particular case become:
\begin{definition}\label{defSspectrum}
Let $\hat{T}\in\boundOP(V_n)$. For $s\in\mathbb{R}^{n+1}$, we set
\[
\Q_{s}(\hat{T}):=\hat{T}^2 - 2\Re(s)\hat{T} + |s|^2\id.
\]
We define the $S$-resolvent set $\rho_S(\hat{T})$ of $\hat{T}$ as
\[
\rho_S(\hat{T}) := \{s\in \mathbb{R}^{n+1}: \Q_{s}(\hat{T})  \text{ is invertible in $\mathcal{B}(V_n)$}\}
\]
and we define the $S$-spectrum $\sigma_S(\hat{T})$ of $\hat{T}$ as
\[
 \sigma_S(\hat{T}) := \mathbb{R}^{n+1}\setminus\rho_S(\hat{T}).
\]
For $s\in\rho_{S}(\hat{T})$, the operator $\Q_{s}(\hat{T})^{-1}\in\boundOP(V_n)$
is  the pseudo-resolvent operator  of $\hat{T}$ at $s$.
\end{definition}
\begin{remark}{\rm
For $\hat{T}\in\boundOP(V_n)$ an equivalent definition of the $S$-spectrum is
\[
\sigma_S(\hat{T}) := \{s\in \mathbb{R}^{n+1}: \Q_{s}(\hat{T})  \text{ is not invertible in $\boundOP(V_n)$}\}
\]
and the  $S$-resolvent set $\rho_S(\hat{T})$ is defined as
\[
 \rho_S(\hat{T}) := \mathbb{R}^{n+1}\setminus\sigma_S(\hat{T}).
\]
}
\end{remark}

\begin{remark}
Observe that in the literature the definition of the $S$-spectrum for $(n+1)$-tuples of noncommuting operators was referred to paravector operators.
In Definition  \ref{defSspectrum}
the $S$-spectrum is defined for operators of the form (\ref{fullyCliop}).
\end{remark}

\begin{remark}
We make a further observation. Consider  a slice monogenic polynomial
$$
P(x)=\sum_{m=0}^Mx^ma_m,\ \ a_m\in \mathbb{R}_n
$$
of order $M$, where $x$ is a paravector.
 We can define the slice monogenic polynomial of the Clifford number $\hat{x}\in \mathbb{R}_n$ by simply replacing the paravector $x$ by $\hat{x}$
 and we get
$$
P(\hat{x})=\sum_{m=0}^M\hat{x}^ma_m,\ \ a_m\in \mathbb{R}_n.
$$
Analogously, we can replace the
the paravector ${x}$ by the bounded full Clifford operator $\hat{T}$ in the polynomial and we get
$$
P(\hat{T})=\sum_{m=0}^M\hat{T}^ma_m,\ \ a_m\in \mathbb{R}_n.
$$
We can repeat the same reasoning for series instead of polynomials.
\end{remark}

Before we define the $S$-functional calculus, we show that the procedure is actually meaningful
because it is consistent with functions of $\hat{T}$ that we can define explicitly, that is with the definition of polynomials of $\hat{T}$. In fact this is a consequence of the general theory (see Section 3 for the notations).

\begin{lemma}\label{PolyT}
Let $\hat{T}\in\boundOP(V_n)$, let $m\in  \mathbb{N}\cup\{0\}$ and let $U\subset \mathbb{R}^{n+1}$
be a bounded slice Cauchy domain with $\sigma_{S}(\hat{T})\subset U$. For any imaginary unit $j\in\mathbb{S}$ we set $ds_j=ds(-j)$. So we have
\begin{align*}
\hat{T}^m =& \frac{1}{2\pi}\int_{\partial(U\cap\mathbb{C}_j)}S_L^{-1}(s,\hat{T})\,ds_j\,s^m
= \frac{1}{2\pi}\int_{\partial(U\cap\mathbb{C}_j)}s^m\,ds_j\,S_R^{-1}(s,\hat{T}),
 \end{align*}
 where $S_L^{-1}$ and $S_R^{-1}$ are as in Definition \ref{ResolventRegularAAA}.
\end{lemma}

\begin{definition}\label{def49}
Let $\hat{T}\in\mathcal{B}(V_n)$.
We denote by $\mathcal{SM}_L(\sigma_{S}(\hat{T}))$,
$\mathcal{SM}_R(\sigma_{S}(\hat{T}))$
and $\mathcal{N}(\sigma_{S}(\hat{T}))$
the set of all  left, right and intrinsic  slice monogenic functions on $U$ where
 where $U$ is any axially symmetric domain such that
$\sigma_{S}(\hat T)\subset U$, $\overline{U}\subset\dom(f)$ and $\dom(f)$ is the domain of the function $f$.
\end{definition}

\begin{definition}[$S$-functional calculus for full Clifford operators]\label{SCalc2}
Let $\hat{T}\in\boundOP(V_n)$. For any imaginary unit $j\in\mathbb{S}$ we set $ds_j=ds(-j)$. We define
\begin{equation}\label{SCalcL}
f(\hat{T}) := \frac{1}{2\pi}\int_{\partial(U\cap\mathbb{C}_j)}S_L^{-1}(s,\hat{T})\,ds_j\,f(s),
 \ \ {\rm for \ any} \ \  f\in\mathcal{SM}_L(\sigma_S(\hat{T})),
\end{equation}
and

\begin{equation}\label{SCalcR}
f(\hat{T}) := \frac{1}{2\pi}\int_{\partial(U\cap\mathbb{C}_j)}f(s)\,ds_j\,S_R^{-1}(s,\hat{T}),
 \ \ {\rm for \ any} \ \  f\in\mathcal{SM}_R(\sigma_S(\hat{T})),
\end{equation}
where $U$ is as in Definition \ref{def49}.
\end{definition}
Theorem~\ref{TGenPolySSS} shows that the $S$-functional calculus is meaningful because it is consistent with polynomials of $\hat{T}$. As the next crucial result shows, the $S$-functional calculus is well-defined because the integrals do not depend on the open set $U$ that contain the $S$-spectrum and
 on  $j\in\mathbb{S}$.

\begin{theorem}\label{OperatorInts}
Let $\hat{T}\in\boundOP(V_n)$. For any $f\in\mathcal{SM}_L(\sigma_{S}(\hat{T}))$, the integral in \eqref{SCalcL} that defines the operator $f(\hat{T})$ is independent of the choice of the slice Cauchy domain $U$ and of the imaginary unit $j\in\mathbb{S}$. Similarly, for any  $f\in\mathcal{SM}_R(\sigma_{S}(\hat{T}))$, the integral in \eqref{SCalcR} that defines the operator $f(\hat{T})$ is also independent of the choice of $U$ and $j\in\mathbb{S}$.
\end{theorem}
\begin{proof}
It is a particular case of Theorem \ref{OperatorIntsSSS}.
\end{proof}

\begin{theorem}\label{SCalcBimon}
 Let $\hat{T}\in\boundOP(V_n)$. If $f\in\mathcal{N}(\sigma_S(\hat{T}))$,
 then both versions of $S$-functional calculus give the same operator $f(\hat{T})$. Precisely, we have
\begin{equation*}\label{SCalcBimonEQ2}
f(\hat{T})=\frac{1}{2\pi}\int_{\partial(U\cap\mathbb{C}_j)}S_L^{-1}(s,\hat{T})\, ds_j\, f(s)= \frac{1}{2\pi}\int_{\partial(U\cap\mathbb{C}_j)}f(s)\,ds_j\,S_R^{-1}(s,\hat{T}).
\end{equation*}
\end{theorem}
\begin{proof}
It is a particular case of Theorem \ref{SCalcBimonSSS}.
\end{proof}

\section{Some properties of the $S$-functional calculus for bounded Clifford operators}
A direct consequence of the definition of the $S$-functional calculus shows that

\begin{lemma}\label{SCalcLinear}
Let $\hat{T}\in\boundOP(V_n)$.

(I)
 If $f,g\in\mathcal{SM}_{L}(\sigma_{S}(\hat{T}))$ and $a\in\mathbb{R}_n$, then
$$
(f+g)(\hat{T}) = f(\hat{T})  + g(\hat{T})\qquad\text{and}\qquad (fa)(\hat{T}) = f(\hat{T})a.
$$

(II) If $f,g\in\mathcal{SM}_{R}(\sigma_{S}(\hat{T}))$ and $a\in\mathbb{R}_n$, then
$$
(f+g)(\hat{T}) = f(\hat{T}) + g(\hat{T})\qquad\text{and}\qquad (af)(\hat{T}) = af(\hat{T}).
$$
\end{lemma}

The following lemma is important when proving the product rule. The proof can be carried out in a similar manner to the quaternionic case, see \cite{acgs}.

\begin{lemma}\label{HelpProdRule}
Let $B\in\boundOP(V_n)$. Assume that $f$ is an intrinsic slice  monogenic function and $U$ is a bounded slice Cauchy domain with $\overline{U}\subset \dom(f)$, then
\[
\frac{1}{2\pi}\int_{\partial(U\cap\cc_j)}f(s)\, ds_j\, (\overline{s}B-Bq)(q^2-2\Re(s)q+|s|^2)^{-1}=Bf(q)
\]
for any $q\in$ $U$ and any $j\in \mathbb{S}$.
\end{lemma}

\begin{theorem}\label{ProdRule}
Let $\hat{T}\in\boundOP(V_n)$.

(I) {\rm (The product rule)}.
Let   $f\in\mathcal{N}(\sigma_S(\hat{T}))$ and $g\in  \mathcal{SM}_{L}(\sigma_S(\hat{T}))$ or let $f\in\mathcal{SH}_{R}(\sigma_S(\hat{T}))$ and $g\in\mathcal{N}(\sigma_S(\hat{T}))$.  Then
$$
(f g)(\hat{T})=f(\hat{T})g(\hat{T}).
$$

(II)
Let $f\in\mathcal{N}(\sigma_S(\hat{T}))$. If $f^{-1}\in\mathcal{N}(\sigma_S(\hat{T}))$, then $f(\hat{T})$ is invertible and $f(\hat{T})^{-1} = f^{-1}(\hat{T})$.
\end{theorem}

\begin{proof}

(I) {\rm The product rule}.
In the following we check the key facts to show that replacing the paravector operator $T$ by the full Clifford operator $\hat{T}$ the proof remain valid.
Let $f\in\mathcal{N}(\sigma_{S}(\hat{T}))$, let $g\in\mathcal{SM}_{L}(\sigma_{S}(\hat{T}))$ and let $U_q$ and $U_s$ be bounded slice Cauchy domains that contain $\sigma_{S}(\hat{T})$ such that
$\overline{U_q}\subset U_s $ and  $\overline{U_s}\subset \dom(f)\cap\dom(g)$. The subscripts $q$ and $s$ refer to  the respective variables of integration in the next computation. Let $j\in\mathbb{S}$, $\Gamma_s := \partial(U_s \cap\cc_j)$ and $\Gamma_q := \partial(U_q\cap\cc_j)$.   Theorem~\ref{SCalcBimon} allows to express $f(\hat{T})$ using both the left and the right $S$-resolvent operator as:
\[
\begin{split}
f(\hat{T}) g(\hat{T})=&
\frac{1}{2\pi } \int_{\Gamma_s} \, f(s)\,ds_j \,S_R^{-1} (s,\hat{T}) \frac{1}{2\pi }\int_{\Gamma_q} \,S_L^{-1} (q,\hat{T})\,dq_j \, g(q).
\end{split}
\]
Using the $S$-resolvent equation we get
\begin{align*}
f(\hat{T})g(\hat{T})=&\frac{1}{(2\pi)^2 }\int_{\Gamma_s} f(s)\,ds_j \int_{\Gamma_q} S_R^{-1}(s,\hat{T})q\Q_{s}(q)^{-1}\,dq_j\, g(q) \\
& -\frac{1}{(2\pi)^2 }\int_{\Gamma_s} f(s)\,ds_j \int_{\Gamma_q}S_L^{-1}(q,\hat{T})q\Q_{s}(q)^{-1}\,dq_j\, g(q)\\
&-\frac{1}{(2\pi)^2 }\int_{\Gamma_s} f(s)\,ds_j \int_{\Gamma_q}\overline{s}S_R^{-1}(s,\hat{T})\Q_{s}(q)^{-1} \,dq_j\, g(q)\\
&+\frac{1}{(2\pi)^2 }\int_{\Gamma_s} f(s)\,ds_j \int_{\Gamma_q}\overline{s}S_L^{-1}(q,\hat{T})\Q_{s}(q)^{-1}\, dq_j\, g(q),
 \end{align*}
 where $\Q_{s}(q)^{-1}:=(q^2-2\Re(s)q+|s|^2)^{-1}$.
Then we have
\begin{align*}
&\frac{1}{(2\pi)^2 }\int_{\Gamma_s} f(s)\,ds_j \int_{\Gamma_q}S_R^{-1}(s,\hat{T})q\Q_{s}(q)^{-1}dq_j\, g(q)\\
=&\frac{1}{(2\pi)^2 }\int_{\Gamma_s} f(s)\,ds_j\,S_R^{-1}(s,\hat{T})\left[\int_{\Gamma_q}
q\Q_{s}(q)^{-1}\,dq_j\, g(q)\right]=0
\end{align*}
and
\[
\begin{split}
&-\frac{1}{(2\pi)^2 }\int_{\Gamma_s} f(s)\,ds_j
\left[\int_{\Gamma_q}\overline{s}\,S_R^{-1}(s,\hat{T})\Q_{s}(q)^{-1}
 dq_j\, g(q)\right]
\\
=
&
-\frac{1}{(2\pi)^2 }\int_{\Gamma_s} f(s)\,ds_j\,
\overline{s}\,S_R^{-1}(s,\hat{T})
\left[\int_{\Gamma_q} \Q_{s}(q)^{-1} \, dq_j\, g(q)\right]= 0
\end{split}
\]
by Cauchy's integral theorem since the functions $\Q_{s}(q)^{-1}$ and $q\Q_{s}(q)^{-1}$ are for any $s\in \Gamma_s$ right slice monogenic on an open set that contains $\overline{U_q}$ since $\overline{U_q}\subset U_s$.
So we can write
\begin{align*}
f(\hat{T}) g(\hat{T})
&=-\frac{1}{(2\pi)^2 }\int_{\Gamma_s} f(s)\,ds_j \int_{\Gamma_q}S_L^{-1}(q,\hat{T})q\Q_{s}(q)^{-1}\,dq_j\, g(q)
\\
&+\frac{1}{(2\pi)^2 }\int_{\Gamma_s} f(s)\,ds_j\int_{\Gamma_q}\overline{s}S_L^{-1}(q,\hat{T})\Q_{s}(q)^{-1}\, dq_j\, g(q)
 \\
 &=\frac{1}{(2\pi)^2 }\int_{\Gamma_s} \int_{\Gamma_q} f(s)\,ds_j\left[\overline{s}S_L^{-1}(q,\hat{T})-S_L^{-1}(q,\hat{T})q\right]\Q_{s}(q)^{-1}dq_j\, g(q).
\end{align*}
The latter integrand is continuous and hence bounded on $\Gamma_s\times\Gamma_q$. Using Fubini's theorem to change the order of integration we obtain
\begin{align*}
f(\hat{T})g(\hat{T})
=\frac{1}{(2\pi)^2 }\int_{\Gamma_q}\left[\int_{\Gamma_s} f(s)\,ds_j\,[\overline{s}S_L^{-1}(q,\hat{T})-S_L^{-1}(q,\hat{T})q] \Q_{s}(q)^{-1}\right]\,dq_j\, g(q).
\end{align*}
By Lemma \ref{HelpProdRule} with $B = S_L^{-1}(q,\hat{T})$, we get
$$
 f(\hat{T}) g(\hat{T})=\frac{1}{2\pi } \int_{\Gamma_q}S_L^{-1}(q,\hat{T})\,dq_j\, f(q) g(q) = (fg)(\hat{T}).
 $$

The product rule for the $S$-functional calculus for right-slice hyperholomorphic functions can be shown with similar computations.

\medskip
Point (II).
From the product rule in point (I), we deduce that
\[
\id = 1(\hat{T}) = \left(ff^{-1}\right)(\hat{T}) = f(\hat{T})f^{-1}(\hat{T})
\]
if we consider $f$ and $f^{-1}$ as left-slice monogenic functions and
\[
\id = 1(\hat{T}) = \left(f^{-1}f\right)(\hat{T}) = f^{-1}(\hat{T}) f(\hat{T})
\]
if we consider them as right-slice monogenic functions. We conclude that $f(\hat{T})$ is invertible with $f^{-1}(\hat{T}) = f(\hat{T})^{-1}$.

\end{proof}

The $S$-functional calculus for Clifford operators $\hat{T}$  defines the Clifford-Riesz projectors. Thus we can identify invariant subspaces of $\hat{T}$ that are associated with sets of spectral values.
\begin{theorem}[Clifford-Riesz projectors]\label{RieszProj}
Let $\hat{T}\in\boundOP(V_n)$ and assume that $\sigma_S(\hat{T})= \sigma_{1}\cup \sigma_{2}$ with
$$
\dist( \sigma_{1},\sigma_{2})>0.
$$
We choose an open axially symmetric set $O$ with $\sigma_1\subset O$ and $\overline{O}\cap \sigma_2 = \emptyset$ and define $\chi_{\sigma_1}(s) = 1$ for $s\in O$ and $\chi_{\sigma_2}(s) = 0$ for $s\notin O$. Then $\chi_{\sigma_1}\in\mathcal{N}(\sigma_S(T))$ and
\[
P_{\sigma_1}: = \chi_{\sigma_1}(\hat{T}) = \frac{1}{2\pi}\int_{\partial(O\cap\cc_{j}) } S_{L}^{-1}(s,T)\,ds_j
\]
is a continuous projection that commutes with $\hat{T}$. Hence, $P_{\sigma_1}V_n $ is a right linear subspace of $V_n$ that is invariant under $\hat{T}$.
\end{theorem}
\begin{proof}
Since the function $\chi_{\sigma_{1}}$ belongs to $\mathcal{N}(\sigma_{S}(T))$ by Theorem~\ref{ProdRule}, we have
\[
P_{\sigma_1}^2 = \chi_{\sigma_1}(\hat{T})\chi_{\sigma_1}(\hat{T}) = (\chi_{\sigma_1}\chi_{\sigma_1})(\hat{T}) = \chi_{\sigma_1}(\hat{T}) = P_{\sigma_1},
\]
and $P_{\sigma_1}$ is a projection in $\boundOP(V_n)$. Since it is right-linear, its range $P_{\sigma_1}V_n$ is a closed right linear subspace of $V_n$. Moreover, we have
\[
\hat{T}P_{\sigma_1} = s(\hat{T}) \chi_{\sigma_1}(\hat{T}) = (s\chi_{\sigma_1})(\hat{T})  = (\chi_{\sigma_1}s)(\hat{T}) = \chi_{\sigma_1}(\hat{T}) s(\hat{T}) = P_{\sigma_1}\hat{T}.
\]
Thus
\[
\hat{T}v = \hat{T}P_{\sigma_1}v = P_{\sigma_1}\hat{T}v\qquad \text{for all }v \in P_{\sigma_1}V_n
\]
hence $P_{\sigma_1}V_n$ is invariant under $\hat{T}$.

\end{proof}

The spectral mapping theorem does not hold for arbitrary slice monogenic
functions, but only for those slice monogenic functions that preserve the fundamental geometric property of the $S$-spectrum, namely its axially symmetry (see Proposition \ref{AxSymSpec}). So we are limited to intrinsic slice monogenic functions that are the only functions that preserve this property.

\begin{theorem}[The Spectral Mapping Theorem]\label{SpectralMapping}
Let  $\hat{T}\in\boundOP(V_n)$ and let $f\in\mathcal{N}(\sigma_S(\hat{T}))$. Then
\[
\sigma_S(f(\hat{T})) = f(\sigma_S(\hat{T})) = \{f(s):s\in\sigma_S(\hat{T})\}.
\]
\end{theorem}
\begin{proof}
Let $U$ be a bounded slice Cauchy domain as in Definition \ref{def49} and let $ s = u + j v\in\sigma_S(T)$. For $q\in U\setminus[ s ]$, we define
$$
\tilde{g}(q) = (q^2 - 2\Re( s )q + | s |^2)^{-1}(f(q)^2 - 2 \Re(f( s ))f(q) + |f( s )|^2).
$$
This function is the product of
\[
q\mapsto f(q)^2 - 2 \Re(f( s ))f(q) + |f( s )|^2,
\]
which is intrinsic and slice monogenic, with the rational intrinsic function $(q^2 - 2\Re(s)q + | s |^2)^{-1}$. So $\tilde g$ is an intrinsic slice monogenic function where it is defined, namely $\tilde g$ belongs to $\mathcal{N}(U)\setminus[ s ]$.

The function $\tilde{g}$ extends to the function $g\in\mathcal{N}(U)$ defined by
\[
 g(q) = \begin{cases} \tilde{g}(q)&\text{if } q\in U\setminus [ s ],\\ \left(\sderiv f( s )\right)^2 &\text{if }q= s
\end{cases}
\]
where $\sderiv f( s )$ is the slice derivative of $f$. Obviously, $g$ is an intrinsic slice function and $g_i = g|_{U\cap\cc_i}$ is holomorphic on $U\cap\mathbb{C}_i$ for any $i\in\mathbb{S}$, so the function $g$ is also slice monogenic.
The product rule implies
\[
 f(\hat{T})^2 - 2\Re(f( s )) f(\hat{T}) + |f( s )|^2\id = (\hat{T}^2 - 2\Re( s )\hat{T} + | s |^2\id)g(\hat{T}).
\]
If the operator $f(\hat{T})^2 - 2\Re(f( s )) f(\hat{T}) + |f( s )|^2\id$ was invertible, then
 $$
 g(\hat{T})(f(\hat{T})^2 - 2\Re(f( s )) f(\hat{T}) + |f( s )|^2\id)^{-1}
 $$
would therefore be the inverse of $\hat{T}^2 - 2\Re( s )\hat{T} + | s |^2\id$. But since we assumed $ s \in\sigma_S(\hat{T})$, this is impossible and hence $f( s )\in \sigma_S(f(\hat{T}))$. Thus
\[
f(\sigma_{S}(\hat{T}))\subseteq \sigma_{S}(f(\hat{T})).
\]
If $ s  \notin f(\sigma_S (\hat{T}))$, then we can consider the function
\[
 h(q) := (f^2(q) - 2\Re( s ) f(q) + | s |^2)^{-1}.
 \]
which is an intrinsic slice monogenic function. Its singularities are the spheres $[q]\subset U$ such that $f([q]) = [f(q)] = [s]$. But since we assumed $s\notin f(\sigma_{S}(T))$, $h$ does not have singularities on $\sigma_{S}(\hat{T})$, $h\in\mathcal{N}(\sigma_{S}(\hat{T}))$ and
Theorem \ref{ProdRule} point (II)
implies
\[
h(\hat{T}) = \left(f(\hat{T})^2 - 2\Re(s)f(\hat{T}) + |s|^2\right)^{-1}\in\boundOP(V_n).
\]
We deduce that $s\in\rho_{S}(\hat{T})$ and
\[
\sigma_{S}(f(\hat{T})) \subseteq f(\sigma_{S}(\hat{T})).
\]
\end{proof}

Now we show the spectral mapping theorem for the $S$-functional calculus of a Clifford operator. The next definition is standard:
\begin{definition}

Let $\hat{T}\in\boundOP(V_n)$. Then the $S$-spectral radius of $\hat{T}$ is defined to be the nonnegative real number
$$
r_S(\hat{T}) := \sup\{|s|: s\in\sigma_S(\hat{T})\}.
$$
\end{definition}
\begin{theorem}\label{GelfandSpecR}
For $\hat{T}\in\boundOP(V_n)$, we have
$$
r_S(\hat{T}) = \lim_{m\to + \infty}\|\hat{T}^m\|^{\frac{1}{m}}.
$$
\end{theorem}
\begin{proof}
The function defined by $q \mapsto q^{-1}$ is evidently intrinsic and slice monogenic, so we can consider the composition function $q\mapsto S_L^{-1}(q^{-1},\hat{T})$ which is slice monogenic on the set
\[
U:=\linebreak[2]\{ q\in\mathbb{R}^{n+1}: q^{-1}\in\rho_{S}(\hat{T})\}.
\]
The set $U$ contains the ball $B_{1/r_S(T)}(0)$, in fact $\mathbb{R}^{n+1}\setminus B_{r_S(T)}(0)\subset \rho_{S}(\hat{T})$, and so the function $S_L^{-1}(q^{-1},\hat{T})$ admits a power series expansion at $0$ that converges on $B_{1/r_S(\hat{T})}(0)$:
\[
S_L^{-1}\left(q^{-1},\hat{T}\right) = \sum_{m=0}^{+\infty} \hat{T}^m q^{m+1}, \qquad |q| < \frac{1}{r_S(\hat{T})}.
\]
For $s$ such that $|s|>r_S(T)$, we have $\left\| \hat{T}^ms^{-m-1}\right\| \to 0$ as $m\to + \infty$ because the above series converges. In particular
\[
 C(s) = \sup_{m\in\mathbb{N}} \|\hat{T}^ms^{-m-1}\| < + \infty.
 \]
Therefore
\begin{align*}
\limsup_{m\to+\infty}\|\hat{T}^m\|^{\frac{1}{m}}\frac{1}{|s|}= \limsup_{m\to+\infty} \|\hat{T}^m\|^{\frac{1}{m}}|s|^{-\frac{m+1}{m}}
= \limsup_{m\to+\infty}\|\hat{T}^ms^{-m-1}\|^{\frac{1}{m}}
\leq \limsup_{m\to+\infty} C(s)^{\frac{1}{m}} = 1,
\end{align*}
and hence
$
\limsup_{m\to+\infty}\|\hat{T}^m\|^{\frac{1}{m}}\leq|s|.
$
 Since $s$ satisfies $|s|>r_S(\hat{T})$, we obtain
$$
\limsup_{m\to+\infty}\|\hat{T}^m\|^{\frac{1}{m}}\leq r_S(\hat{T}).
$$
Moreover, Theorem \ref{SpectralMapping} implies
$
\sigma_S(\hat{T}^m) = \sigma_S(\hat{T})^m,
$
so  Theorem~\ref{PropSpec} yields that
\begin{align*}
 r_S(\hat{T})^m = \sup\{|s|^m:s\in\sigma_S(T)\}
 = \sup\{|s|:s\in \sigma_S(\hat{T}^m)\} = r_S(\hat{T}^m)\leq \|\hat{T}^m\|
 \end{align*}
for any $n\in\mathbb{N}$. Therefore, we get
\begin{equation}\label{SRP}
r_S(\hat{T}) \leq\liminf_{m\to +\infty} \|\hat{T}^m\|^{\frac{1}{m}} \leq \limsup_{m\to +\infty}\|\hat{T}^m\|^{\frac{1}{m}}\leq r_S(\hat{T})
\end{equation}
and $r_S(\hat{T}) = \lim_{m\to\infty} \|\hat{T}^m\|^{\frac{1}{m}}$, where \eqref{SRP} also implies the existence of the limit.
\end{proof}

Given two slice monogenic functions, in general it is not possible to define their composition. In order to do, the second function has to be intrinsic.
The spectral mapping theorem allows to generalize the composition rule.
\begin{theorem}[Composition rule]\label{Composition}
Let $\hat{T}\in\boundOP(V_n)$ and let $f\in\mathcal{N}(\sigma_{S}(\hat{T}))$. If $g\in\mathcal{SM}_L(\sigma_{S}(f(\hat{T}))$ then $g\circ f \in \mathcal{SM}_L(\sigma_{S}(\hat{T}))$ and if $g\in\mathcal{SM}_R(f(\sigma_{S}(\hat{T})))$ then $g\circ f \in \mathcal{SM}_R(\sigma_{S}(\hat{T}))$. In both cases
\[
 g(f(\hat{T})) = (g\circ f) (\hat{T}).
 \]
\end{theorem}
\begin{proof}
If $g\in\mathcal{SM}_L(f(\sigma_{S}(\hat{T})))$, then $\dom(g)$ can be chosen open and axially symmetric. Since $f$ is continuous and intrinsic, the inverse image of any open axially symmetric set under $f$ is again open and axially symmetric.  Thus $f^{-1}(\dom(g))$ is an axially symmetric open set containing $\sigma_{S}(\hat{T})$ as $f(\sigma_{S}(\hat{T})) = \sigma_{S}(f(\hat{T}))\subset \dom(g)$ by Theorem~\ref{SpectralMapping}.

Since $f\in\mathcal{N}$ the composition $g\circ f$ is a left slice monogenic function with domain $f^{-1}(\dom(g))$ and so $g\circ f\in\mathcal{SM}_L(\sigma_{S}(\hat{T}))$.

Let $U$ be a bounded slice Cauchy domain such that $\sigma_{S}(\hat{T})\subset U$ and $\overline{U}\subset \dom(f)$ and let $W$ be another bounded slice  Cauchy domain such that $\sigma_{S}(\hat{T})\subset \overline{f(U)} \subset W$ and $\overline{W}\subset \dom(g)$.  Since $f$ is intrinsic, the map $s\mapsto S_L^{-1}(q, f(s))$ is left slice monogenic on
\[
\{s\in\dom(f): f(s)\notin [q] \} = \{s\in\dom(f): q\notin[f(s)] \}.
\]
If $q\notin\sigma_S(f(\hat{T})) = f(\sigma_{S}(\hat{T}))$, then $s\mapsto S_L^{-1}(q,f(s))$ belongs to $\mathcal{SM}_L(\sigma_{S}(\hat{T}))$. By the properties of the $S$-functional calculus, we have
\begin{align*}
S_L^{-1}( q ,f(\hat{T})) =& -\Q_{q}(f(\hat{T}))^{-1}( f(\hat{T}) - \overline{ q }\id)
\\
=& \frac{1}{2\pi}\int_{\partial(U\cap\mathbb{C}_j)} S_L^{-1}(s,\hat{T})\,ds_j\,\left[ -\Q_{q}(f(s))^{-1}(f(s) - \overline{ q })\right]
\\
= &\frac{1}{2\pi}\int_{\partial(U\cap\mathbb{C}_j)}S_L^{-1}(s,\hat{T})\,ds_j\,S_L^{-1}( q ,f(s))
\end{align*}
with $\Q_{s}(f(s))^{-1} = (f(s)^2 - 2\Re(q)f(s) + |q|^2)^{-1}$ and an arbitrary $j\in\mathbb{S}$.
Therefore
\begin{align*}
g(f(\hat{T})) =& \frac{1}{2\pi}\int_{\partial(W\cap\mathbb{C}_j)}S_L^{-1}( q , f(\hat{T}))\, d q _j\,g( q )
\\
=& \frac{1}{2\pi}\int_{\partial(W\cap\mathbb{C}_j)}\left[\frac{1}{2\pi}
\int_{\partial(U\cap\mathbb{C}_j)}S_{L}^{-1}(s,\hat{T})\, ds_j \, S_L^{-1}( q , f(s))\right] \,d q _j\,g( q ).
\end{align*}
Since the latter integrand is continuous and hence bounded on the compact set $\partial(W\cap\mathbb{C}_j)\times \partial (U\cap \mathbb{C}_j)$, we can apply Fubini's theorem to change the order of integration and obtain
\begin{align*}
g(f(\hat{T}))
=& \frac{1}{2\pi}\int_{\partial(U \cap\mathbb{C}_j)}S_{L}^{-1}(s,\hat{T})\,ds_j\, \left[ \frac{1}{2\pi}\int_{\partial(W \cap\mathbb{C}_j)}S_L^{-1}( p , f(s))\,d p _j\,g( p )\right]
\\
=& \frac{1}{2\pi}\int_{\partial(U\cap\mathbb{C}_j)}S_{L}^{-1}(s,\hat{T})\,ds_j\,g(f(s))
\\
=& \frac{1}{2\pi}\int_{\partial(U\cap\mathbb{C}_j)}S_{L}^{-1}(s,\hat{T})\,ds_j\,(g\circ f)(s) = (g\circ f)(\hat{T}).
\end{align*}
\end{proof}

\section{Noncommuting matrix variables and some final remarks}
The results of the previous section  can be of interest for the community working in free analysis
and in free probability.
In fact, as a particular case,
we can consider $(n+1)$ noncommuting matrices. Precisely,
let $X_j\in\mathbb{R}^{d\times d}$, for $d\in \mathbb{N}$ and let us make the identification
$$
(X_0,X_1,...,X_n)\to \mathbf{X}=\sum_{j=0}^nX_je_j
$$
where $e_1,\ldots,e_n$ are generators of the Clifford algebra $\mathbb R_n$.
Then Theorem \ref{leftrightSSS} becomes:
\begin{theorem}
Let  $\mathbf{X}\in \mathbb{R}^{d\times d}\otimes\mathbb{R}_n$, where $\otimes$ denotes the algebraic tensor product, and $s\in \mathbb{R}^{n+1}$ with $\| \mathbf{X}\| < |s|$.

(I)
The left $S$-resolvent series equals
\[
\sum_{m=0}^{+\infty}\mathbf{X}^ms^{-m-1} = - (\mathbf{X}^2 - 2\Re (s)\mathbf{X} +|s|^2\id_{d\times d})^{-1}(\mathbf{X}-\overline{s}\id_{d\times d}).
\]

(II)
The right $S$-resolvent series equals
\[
 \sum_{m=0}^{+\infty}s^{-m-1}\mathbf{X}^m = - (\mathbf{X}-\overline{s}\id_{d\times d})(\mathbf{X}^2 - 2\Re(s)\mathbf{X} +|s|^2\id_{d\times d})^{-1}.
 \]
\end{theorem}
Then the $S$-spectrum of the noncommuting matrices $(X_0,X_1,...,X_n)$ is defined as:
\begin{definition} Let $\mathbf{X}=\sum_jX_je_j\in \mathbb{R}^{d\times d}\otimes\mathbb{R}_n$ and take $s\,\in\mathbb{R}^{n+1}$.
We define the $S$-spectrum of the  $\mathbf{X}\in \mathbb{R}^{d\times d}\otimes\mathbb{R}_n$ as
$$
\sigma_S(\mathbf{X})=\{s\in \mathbb{R}^{n+1} \ :\  \mathbf{X}^2-2\Re(s)\mathbf{X} +|s|^2\mathcal{I}_{d\times d} \ \ {\rm is \  not \ invertible \  in} \ \mathbb{R}^{d\times d}\otimes\mathbb{R}_n\}
$$
and the $S$-resolvent set as
$$
\rho_S(\mathbf{X})=\mathbb{R}^{n+1}\setminus \sigma_S(\mathbf{X}).
$$
\end{definition}
The resolvent operators are:
\begin{definition}[The $S$-resolvent operators]
Let $\mathbf{X}\in \mathbb{C}^{d\times d}\otimes\mathbb{R}_n$ . For $s\in\rho_S(\mathbf{X})$, we define the {\em left $S$-resolvent operator} as
$$
 S_L^{-1}(s,\mathbf{X}) = - (\mathbf{X}^2 - 2\Re (s)\mathbf{X} +|s|^2\id_{d\times d})^{-1}(\mathbf{X}-\overline{s}\id_{d\times d}),
 $$
and the {\em right $S$-resolvent operator} as
$$
S_R^{-1}(s,\mathbf{X}) = - (\mathbf{X}-\overline{s}\id_{d\times d})(\mathbf{X}^2 - 2\Re(s)\mathbf{X} +|s|^2\id_{d\times d})^{-1}.
$$
\end{definition}

Via the $S$-functional calculus we can define
hyperholomorphic functions of the noncommuting matrices $\mathbf{X}$.
In particular the case of intrinsic functions contains all
special functions that have a power series expansion, e.g., the exponential, sine, cosine, Bessel, more in general hypergeometric functions to name a few.

We conclude this section with some connections of the spectral theory on the $S$-spectrum and other research fields. We mention below some directions in which the $S$-functional calculus has been developed. We wish to stress that the main conclusions in this paper allow one to abstract any result on the $S$-functional calculus below to a setting where one may embed an $n$-tuple $(T_1, \ldots, T_n)$ of bounded operators acting on a Banach real space into an operator acting on a Banach module over $\mathbb{R}_m$, with $m \geq \frac{\log(n)}{\log(2)}$, or more generally on any Banach module over a left complex structure (LSCS) algebra with dimension greater than or equal to $n$ (and not just as a paravector operator).
Note that the choice of the embedding is  highly non-canonical and may be purpose driven.

\begin{remark}
First of all we would like to mention that, in the quaternionic case, there have been important developments in Schur analysis
in the slice hyperholomorphic setting. The material is organized in the book \cite{6ACSBOOK} and in the references therein. The quaternionic $S$-functional calculus is better developed and full treated in \cite{6CKG,6css}.
 Fractional powers of vector operators and applications have been largely investigated in the papers
 \cite{frac4,frac5,frac3} and in the book \cite{6CG}.
  The $H^\infty$-functional calculus was further extended  in \cite{6MILANO} following the book \cite{Haase}.
 For more recent developments associated with quaternionic quantum mechanics see
\cite{santar2,santar3}.
\end{remark}

\begin{remark}
In the Clifford algebra setting, the theory is based on slice monogenic functions which were developed in the book \cite{6css} starting from the paper
\cite{6cssisrael},  and the
 $S$-functional calculus for paravector operators was originally developed in \cite{CSSFUNCANAL}.
More advances in spectral theory in the Clifford setting can be found in \cite{6hinfty} in which the $H^\infty$-functional calculus has been extended to this setting.
\end{remark}

\begin{remark}
 The $F$-functional calculus \cite{6CKG,FCALSOM} is a link between the spectral theory on the $S$-spectrum and the monogenic spectral theory \cite{6MQ,6JM,6JMP,6MP,6qian1}, see also the books \cite{jefferies,BOOKTAO}.
The seminal paper on the $H^\infty$-functional calculus is
\cite{McIntosh:1986}.
The $F$-functional calculus is based
    on the Fueter-Sce-Qian mapping theorem, a natural relation between slice monogenic and monogenic functions, see the books \cite{bookSCE,BOOKTAO}. Another link between slice monogenic functions and the classical monogenic functions can be found in \cite{6Radon} using the method of the Radon transform.

\end{remark}

\end{document}